\documentclass[journal,twoside,web]{ieeecolor}
\usepackage{generic}
\usepackage{cite}
\usepackage{amsmath,amssymb,amsfonts}
\usepackage{algorithmic}
\usepackage{graphicx}
\usepackage{textcomp}

\usepackage{amsthm}
\usepackage[linesnumbered,ruled,vlined]{algorithm2e}
\usepackage{xcolor}

\SetKwInput{KwInput}{Input}                
\SetKwInput{KwOutput}{Output}              

\newcommand{\until}[1]{\{1,\dots, #1\}}

\newcommand{\untilzn}[1]{\{0, 1,\dots, #1\}}

\newcommand{\untilzinf}{\{0, 1,\dotsc\}}
\newcommand{\subscr}[2]{#1_{\textup{#2}}}

\def\BibTeX{{\rm B\kern-.05em{\sc i\kern-.025em b}\kern-.08em
    T\kern-.1667em\lower.7ex\hbox{E}\kern-.125emX}}
\begin{document}
\title{Enhancing network resilience through topological switching}
\author{Fei Chen, Jorge Cort\'es, \IEEEmembership{Fellow, IEEE}, and Sonia Mart{\'\i}nez, \IEEEmembership{Fellow, IEEE}
\thanks{Fei Chen is with the Department of Aeronautics and Astronautics, Massachusetts Institute of Technology, Cambridge, MA 02139, USA (email: feic@mit.edu). This work was conducted when he was at the University of California, San Diego.}
\thanks{Jorge Cort\'es and Sonia Mart{\'\i}nez are with the Department of Mechanical
 and Aerospace Engineering, University of California, San
Diego, CA 92093, USA (e-mail: cortes@ucsd.edu; soniamd@ucsd.edu).}
}

\maketitle

\newtheorem{remark}{Remark}
\newtheorem{definition}{Definition}
\newtheorem{proposition}{Proposition}
\newtheorem{theorem}{Theorem}
\newtheorem{lemma}{Lemma}
\newtheorem{corollary}{Corollary}
\newtheorem{problem}{Problem}
\newtheorem{assumption}{Assumption}

\newcommand{\Real}{\operatorname{Re}}
\newcommand{\diag}[1]{\operatorname{diag}(#1)}
\newcommand{\spa}{\operatorname{span}}
\newcommand{\ve}{\operatorname{vec}}
\newcommand{\rank}{\operatorname{rank}}
\newcommand{\nullity}{\operatorname{nullity}}
\newcommand{\Tr}{\operatorname{Tr}}
\newcommand{\ROS}{\operatorname{ROS}}
\newcommand{\rROS}{\operatorname{rROS}}
\newcommand{\AOone}{\mathop{\mathrm{AO}_1}}
\newcommand{\AOtwo}{\mathop{\mathrm{AO}_2}}
\newcommand{\AOthr}{\mathop{\mathrm{AO}_3}}
\newcommand{\lzeroRec}{\operatorname{\ell_{0}-rec}}
\newcommand{\loneRec}{\operatorname{\ell_{1}-rec}}
\newcommand{\rloneRec}{\operatorname{r-\ell_{1}}}
\newcommand{\rrloneRec}{\operatorname{r^2-\ell_{1}}}
\newcommand{\cvxRec}{\operatorname{cvx-\ell_{1}}}
\newcommand{\amRec}{\operatorname{am-\ell_{1}}}

\newcommand{\optswitch}{\operatorname{OptSwitch}}

\newcommand{\oprocendsymbol}{\hbox{$\bullet$}} 
\newcommand{\oprocend}{\relax\ifmmode\else\unskip\hfill\fi\oprocendsymbol}
\def\eqoprocend{\tag*{$\bullet$}}

\begin{abstract}
This work studies how to preemptively increase the resilience of a network  by means of time-varying topological actuation. 
To do this, we focus on linear dynamical systems that are compatible with a given network, and consider policies that switch periodically between the given one and an alternative, topologically-compatible dynamics. In particular, we seek to solve design problems aimed at finding a) the optimal switching schedule between two preselected topologies, and b) an optimal topology and optimal switching schedule.  By imposing periodicity, we first provide a metric of resilience in terms of the spectral abscissa of the averaged linear time-invariant dynamics. By restricting our policies to commutative networks, we then show how the optimal scheduling problem reduces to a convex optimization, providing a bound on  the net resilience that can be achieved. After this, we find that the optimal, sparse commutative network to switch with is fully disconnected and allocates the spectral sum among the nodes of the network equally. We then impose additional restrictions on topology edge selection, which leads to a biconvex optimization for which certain matrix rank conditions guide the choice of weighting parameters to obtain desirable solutions.
Finally, we provide two methods to solve this problem efficiently (based on a McCormick relaxation, and alternating minimization), and illustrate the results in simulations. 
\end{abstract}

\begin{IEEEkeywords}
Network resilience, time-varying actuation, optimal switching, sparsity-promoting optimization.
\end{IEEEkeywords}

\section{Introduction}
\label{sec:introduction}
Networked systems find broad application in diverse domains such as robotics, social networks~\cite{gatti2013large}, security~\cite{gorodetski2002multi}, and intelligent transportation systems~\cite{burmeister1997application}.  
While significant progress has been made in developing their function,  
networked systems remain highly vulnerable to various forms of attacks and failures. 
As the dependence of infrastructure systems on networked systems grows, so does the harmful impact of attacks, with great social and economic cost.


Here, we study how to make networks more resilient to structural attacks, aimed at disrupting the topology that supports the network operation. In particular, we are interested in quantifying how switching between complementary meshes can reduce the impact of a unknown and fixed topological failure. As switching can lead to instability, this requires the careful design of the time-varying actuation strategy together with the selection of complementary meshes for control.

 The resilience of networks described by linear time-invariant dynamics is often analyzed through metrics derived from the spectral and pseudo-spectral properties of the associated matrix. Specifically, the pseudo-spectral abscissa~\cite{trefethen2020spectra} is a key metric for assessing the sensitivity of a dynamical system to perturbations. 
 Existing algorithms to compute it include the criss-cross algorithms~\cite{lu2017criss,burke2003robust}, which rely on the bisection method~\cite{byers1988bisection}, and~\cite{rostami2015new,guglielmi2011fast}, which provide fast iterations that exploit the sparse structure of matrices. 
The stability radius is a related notion that aims to capture the smallest perturbation required to destabilize a system. The work~\cite{guglielmi2016method} 
approximates the stability radius of sparse matrices leveraging~\cite{rostami2015new}. In~\cite{liu2022iterative}, the authors characterize the analogs of these notions for a class of constrained structured perturbations, proposing iterative algorithms that are capable of tackling networks with hundreds of nodes.

The literature on time-varying network systems is much more limited with most of the focus on achieving stability despite different types of attacks and under intermittent or time-varying communications. Some early work includes~\cite{HSF-SM:16-sicon,MZ-SM:12}, on system stabilization under Denial of Service and reply attacks,~\cite{SS-BG:18} on distributed optimization under adversarial nodes, or~\cite{pasqualetti2013attack} on attack detection and identification. More recently,~\cite{mitra2021distributed} introduces a single-time-scale distributed state estimation algorithm under time-varying graphs and worst-case adversarial attacks, while~\cite{khajenejad2025distributed} proposes a distributed interval estimation algorithm for arbitrary attack isolation. 
In contrast, the work~\cite{SL-SM-JC:23-tac} proposes switching as a defense mechanism against rational adversaries with partial knowledge on an operator's decisions. 


The first contribution of this work is an analysis of how switching between two given topologies can increase the resilience of a given network. By prescribing periodic changes between two linear dynamics, we ease the time-varying system analysis to an averaged linear time-invariant counterpart, which allows us to identify the spectral abscissa of the latter as the resilience metric of interest. Then, we formulate and solve design problems of increased difficulty to optimize this metric. First, we look for an optimal switching schedule between two commutative networks, and show how this problem reduces to a convex optimization. 
In particular, we bound net resilience using the largest eigenvalue of each system matrix together with the eigenvalue of the other matrix associated with the same eigenvector.
Second, we consider the joint design of both a complementary network and the associated optimal schedule. This is done by including a sparsity-promoting term in the problem cost that results into a hard optimization. Interestingly, we are able to show that the best commutative topology  is one that matches a completely disconnected network. After this, we further constrain the problem to limit the set edges that can be modified. We find an algebraic rank condition that establishes when a compatible commutative network exists and obtain biconvex problem formulations to solve the associated problem. Finally, we introduce two efficient algorithms (based on McCormick relaxation and alternating minimization algorithms) to construct the complementary network, and evaluate their effectiveness in simulation. 
\section{Preliminaries}
In this section, we introduce notations used throughout the paper, summarizing main concepts from Floquet theory~\cite{floquet1883equations}. 
\subsubsection{Notation}\label{sub:notation}
In what follows, $I_d$ denotes the $d\times d$ identity matrix. Given a matrix $A =[a_{ij}]\in \mathbb{R}^{n \times n}$, $a_{ij}$ 
denotes its $ij$-th entry, $\Tr(A)$ refers to its trace, and $\ve(A)\in \mathbb{R}^{n^2}$ is the vectorized form of $A$ obtained by stacking its columns. By $\rank(A)$ we denote the rank of $A$, and by $\nullity(A)$ the dimension of its null space.
For a vector $\mathrm{v}=[v_1, v_2,\dots,v_n]^T\in \mathbb{R}^n $, $\diag{\mathrm{v}}$ represents the diagonal matrix with diagonal equal to  $\mathrm{v}$. 
The span of $A_i \in \mathbb{R}^{n \times n}, i \in \until{k}$,  is defined as
$\spa \{A_1, A_2, \dots, A_k\} 
= \left\{ c_1 A_1 + c_2 A_2 + \cdots + c_k A_k \;\middle|\; c_1, c_2, \dots, c_k \in \mathbb{R} \right\}.$ The operation $A \odot B$ represents the element-wise (Hadamard) multiplication of matrices $A$ and $B$. The Kronecker product of $A$ and $B$ is denoted by $A \otimes B$. A row vector of zeros of appropriate dimensions is represented by $0^\top$.
The indicator function $\chi$ is defined such that $\chi(z) = 1$, if $z \neq 0$, and $\chi(z) = 0$, if $z = 0$. With this, we define $ {\|A\|}_{\ell_0}=\sum\nolimits_{i,j=1}^n \chi (a_{ij})$ and, similarly, ${\|A\|}_{\ell_1}=\sum\nolimits_{i,j=1}^n |a_{ij}|$. The operator $\Real$ 
maps a matrix or vector to its entry-wise real part. We denote by $|\mathcal{I}|$ the cardinality of a set $\mathcal{I}$. By $\mathbb{B}_\epsilon:=\{\Delta\in \mathbb{C}^{n\times n}\mid \|\Delta\|_F\leq \epsilon \}$ we represent a closed ball centered at $0$ with radius $\epsilon$ in  $\mathbb{C}^{n\times n}$, and $\|\cdot\|_F$ denotes the Frobenius norm.

\subsubsection{Spectral abscissa and worst-case perturbation}\label{sub:worst-case}
The spectral abscissa of a square matrix $A$ is defined by
\begin{equation}\label{eq:spec_abs}
    \alpha(A):=\max\limits_{\lambda\in\Lambda(A)} \Real \lambda,
\end{equation}
where $\Lambda(A)$ is the spectrum (the set of the eigenvalues) of $A$. Given $\epsilon>0$ and a closed set $\mathcal{H}\subseteq \mathbb{C}^{n\times n}$, the structured $\epsilon$-pseudospectrum of $A$ is 
\begin{equation}\label{eq:stru_spec}
    \Lambda_{\epsilon,\mathcal{H}}(A):=\{\lambda\in\Lambda(A+\Delta),  \Delta\in \mathcal{H} \cap \mathbb{B}_{\epsilon}\},
\end{equation}
which represents the union of the spectrum of $A+\Delta$, for $\Delta \in \mathcal{H} \cap \mathbb{B}_{\epsilon}$.  When 
$\mathcal{H}= \mathbb{C}^{n\times n}$, $\Lambda_{\epsilon,\mathcal{H}}$ reduces to the usual $\epsilon$-pseudospectrum~\cite{guglielmi2011fast}. The structured $\epsilon$-pseudospectral abscissa of $A$ is then expressed as 
\begin{equation}\label{eq:stru_abscissa}
    \alpha_{\epsilon,\mathcal{H}}(A):=\max\limits_{\lambda\in\Lambda_{\epsilon,\mathcal{H}}(A)} \Real \lambda = \max\limits_{\Delta\in \mathcal{H} \cap \mathbb{B}_{\epsilon}}\alpha(A+\Delta).
\end{equation}
We call a maximizer $\subscr{\Delta}{opt}$ of $\alpha_{\epsilon,\mathcal{H}}(A)$ as a worst-case perturbation of $A$ of energy $\epsilon$. 

\subsubsection{Floquet theory}\label{sub:Floquet}
Floquet theory~\cite{floquet1883equations} provides a mechanism to analyze linear time-varying periodic systems by means of an averaged time-invariant counterpart. Briefly, consider 
\begin{equation}\label{eq:ptvs}
  \dot{x}(t)=S(t)x(t),\hspace{4mm} x(t_0)=x_0,
\end{equation}
where $x(t)\in \mathbb{R}^n$, and $S(t)\in \mathbb{R}^{n\times n}$ is piecewise continuous, bounded, and periodic with period $T$, i.e.,  $S(t+T)=S(t)$. Floquet theory states that
\begin{itemize}
    \item There exists a nonsingular matrix $P(t,t_0)$ with $P(t+T,t_0)=P(t,t_0)$, and a constant matrix $Q$ such that 
    \begin{equation} \nonumber
        \Phi(t,t_0)=P(t,t_0)e^{Q(t-t_0)},
    \end{equation}
    where $\Phi(t,t_0)$ is the system state-transition matrix. 
    \item The time-dependent change of coordinates $z(t)=P^{-1}(t,t_0)x(t)$ transforms~\eqref{eq:ptvs} into the following linear time-invariant system
    \begin{equation} \label{eq:LTI_pre}
        \dot{z}(t)=Qz(t),\hspace{4mm} z(t_0)=x_0.
    \end{equation}
\end{itemize}    
Let 
\begin{equation}\label{eq:R_pre}
    R=\Phi(t_0+T,t_0),
\end{equation}
then $Q$ can be expressed as 
\begin{equation}\label{eq:Q_pre}
    Q=\frac{1}{T}\ln{R}.
\end{equation}
The eigenvalues of $Q$, $\lambda_i, i \in \until{n}$, and the eigenvalues of $R$, $\nu_i, i\in \until{n}$, are related as follows
\begin{equation}\nonumber
    \nu_i=e^{\lambda_iT},\hspace{4mm} i\in \until{n}.
\end{equation}
In this way, the linear time-varying system~\eqref{eq:ptvs} is exponentially stable if and only if $R$ is Schur or $Q$ is Hurwitz~\cite{gokccek2004stability}. 
Therefore, the stability of~\eqref{eq:ptvs} can be inferred from that of~\eqref{eq:LTI_pre}.
\section{Problem statement and applications}
In this section, we present a formal statement of the main problems to be addressed in the rest of the manuscript. In addition, we introduce application scenarios that show the relevance of the proposed problems. 

\subsection{System description}
Consider a \textit{network} represented by a directed graph or \textit{digraph} $\mathcal{G}:=(\mathcal{V},\mathcal{E})$, which consists of a vertex set $\mathcal{V}=\until{n}$ and an edge set $\mathcal{E}\subseteq \mathcal{V} \times \mathcal{V}$. 
An in-neighbor of $i \in \mathcal{V}$ is a vertex $j$ such that $(i,j) \in \mathcal{E}$. We denote the set of in-neighbors of $i$ by $\mathcal{N}_i$. By assigning a weight $a_{ij}\in \mathbb{R}$ to each edge $(i,j)\in \mathcal{E}$ and $a_{ij}=0$ if $(i,j)\notin \mathcal{E}$, a compatible network  dynamics is described as the time-invariant system 
\begin{equation}\label{eq:dyna}
    \dot{x}=Ax,
\end{equation}
where $x\in \mathbb{R}^n$ are the stacked states of all agents, and $A=[a_{ij}]\in \mathbb{R}^{n\times n}$ is the weighted adjacency matrix. We assume that the matrix $A$ is Hurwitz, therefore $\alpha(A)<0$. For brevity, we use either $A$ or the underlying digraph $\mathcal{G}$ to refer to a network. 
We are interested in reducing the impact of adversarial attacks on~\eqref{eq:dyna}  represented by multiplicative perturbations  $\Delta$ of the form: 
\begin{equation}\nonumber
    \dot{x}=(A+\Delta)x.
\end{equation}
Since the worst-case perturbation, given by~\eqref{eq:stru_spec} and~\eqref{eq:stru_abscissa}, depends on the network structure, it should be possible to adjust it to enhance its resilience. In particular, the smaller (i.e., more negative) $\alpha(A)$ is, the greater the resilience~\cite{liu2022iterative}. This function is inversely proportional to what can be interpreted as a \textit{distance-to-instability} metric, which is our goal to increase.

\subsection{Problem statement}
To augment the resilience of~\eqref{eq:dyna}, we will enlarge our system with time-varying actuation, as described by 
\begin{equation}\nonumber
    \dot{x}=S(t)x.
\end{equation}
This time-varying actuation will be limited to a few actions as specified by a finite set of topological network changes. In particular, we are interested in understanding
a) when it is possible to achieve a net resilience-improvement by topological switching, and b) how to design such new network topology and switching schedule to realize this benefit.  

To state these goals more precisely, consider two networks represented by weighted matrices $A, B\in \mathbb{R}^{n\times n}$ and associated digraphs $\mathcal{G}_A:=(\mathcal{V}_A,\mathcal{E}_A)$ and $\mathcal{G}_B:=(\mathcal{V}_B,\mathcal{E}_B)$, respectively.  A switched system defined by $A,B$ is given by
\begin{equation}\label{eq:switch}
\begin{aligned}
    \dot{x}(t)&=S(t)x(t)\\&=\begin{cases}
      Ax(t), & t_{i,1}+lT\leq t <t_{i,2}+lT,\\
      Bx(t), & t_{i,2}+lT\leq t <t_{i+1,1}+lT,\\
    \end{cases}  
\end{aligned}    
\end{equation}
where $i\in \until{m}$ and $m$ represents the occurrences of $A$ and $B$ during a time period $T$, and $t_{1,1}=0, t_{m+1,1}=T$; $l\in \untilzinf$. Furthermore, we define $\Delta t_{i,1}=t_{i,2}-t_{i,1}$ (resp.~$\Delta t_{i,2}=t_{i+1,1}-t_{i,2}$), $i\in \until{m}$, as the dwell time when $A$ (resp.~$B$) is active. 

The resilience of the switched system~\eqref{eq:switch} will be characterized via the spectral abscissa of an associated dynamical system. In the following, we denote this function as $\alpha(S(t))$. 


We now formally state the problems considered here.
\begin{problem}{\rm
    Given two 
    networks $A$, $B$ defining the switched system~\eqref{eq:switch}, quantify $\alpha(S(t))$ and derive an optimal switching between $A$, $B$ such that $\alpha(S(t))$ is minimized. In other words, solve for the optimization
    \begin{equation} \label{eq:optimiation_ori}
\begin{aligned}
\min_{\Delta t_{i,1}, \Delta t_{i,2}} \quad & \alpha(S(t))\\
\textrm{s.t.} \quad & 0< \Delta t_{i,1} < T,\\
& 0< \Delta t_{i,2} < T,\quad i\in \until{m},\\
& \sum_{i=1}^m (\Delta t_{i,1}+ \Delta t_{i,2})=T.\\
\end{aligned}
\end{equation}}
\oprocend
\end{problem}
\begin{problem}{\rm
Given a network $A$, optimize the structural changes of $A$ to obtain a resilience-enhancing complementary network $B$ such that $\alpha(S(t))<\min(\alpha(A),\alpha(B))$, via time-varying actuation as in~\eqref{eq:switch}. }\oprocend
\end{problem}

\subsection{Physical application scenarios}\label{sec:application}
This section presents some main applications of the resilience-enhancing topology-switching problems introduced above, showing how these relate to 
the abstracted dynamics~\eqref{eq:dyna}. Section~\ref{sec:simulations} will revisit these examples in simulation. 
\subsubsection{Formation control}
Consider a group of agents $i\in \until{n}$, 
and divide these into sets $\mathcal{V}_g$ and $\mathcal{V}_l$, $\mathcal{V}_g \cup \mathcal{V}_l = \until{n}$, where $\mathcal{V}_g$ consists of agents with privileged access to global information (e.g.~by interacting with a fusion center). 
Let $K_i \in \mathbb{R}^{d\times d}$ be a local gain matrix, for $i \in \mathcal{V}$, and $w_{ij}$ a non-zero positive weight for an edge $(i,j)$ in the agent-interaction graph $\mathcal{G} = (\mathcal{V}_g\cup \mathcal{V}_l, \mathcal{E})$. The associated  weighted Laplacian matrix 
$L_w=[l_{ij}]\in \mathbb{R}^{n\times n}$, is given by $l_{ij}=-w_{ij}\in \mathbb{R}$, if $i\neq j$, and $l_{ii}=\sum_{k = 1 }^n w_{ik}$, for $i \in \until{n}.$ 

Let $|\mathcal{V}_g|=n_g, $ $|\mathcal{V}_l|=n_l$, with $n_g + n_l =n$. Denote by $z_g\in \mathbb{R}^{d n_g}$ and $z_l\in \mathbb{R}^{d n_l}$ the stacked agents' states in $\mathcal{V}_g$ and $\mathcal{V}_l$, respectively. Let the weighted Laplacian matrix $L_w$ be divided into blocks $L_{gg}, L_{gl}, L_{gl}, L_{ll}$,  describing the interactions within and between the groups of agents $\mathcal{V}_g$ and $\mathcal{V}_l$, respectively. Similarly, define the block-diagonal gain matrices  $K_g = \diag{K_i\,|\, i \in \mathcal{V}_g}$, and $K_l = \diag{K_i\,|\, i\in \mathcal{V}_l}$. A generalized formation-control dynamics is given as
\begin{equation}\label{eq:s1_le}
    \dot{z}_g=A_gz_g+A_{gl}z_l+r^\star(t),
\end{equation}
\begin{equation}\label{eq:s1_fo}
    \dot{z}_l=A_lz_l+A_{lg}z_g,
\end{equation}
where $A_{g} =  K_{g} +  L_{gg} \otimes I_d$, $A_{gl} =  L_{gl} \otimes I_d $, $A_l =  K_{l} + L_{ll} \otimes I_d$, and $A_{lg} = L_{lg} \otimes I_d$. 
The time-varying mapping $r^\star:\mathbb{R}_{\geq 0}\rightarrow \mathbb{R}^{d n_g}$ encodes the information that only agents in $\mathcal{V}_g$ have access to. By stacking~\eqref{eq:s1_le} and~\eqref{eq:s1_fo}, defining $A$ in a block-wise manner via $A_{gg}, A_{gl}, A_{lg}, A_{gg}$, and setting $f(t) = [r(t)^\top, 0^\top]^\top \in \mathbb{R}^{dn}$, we can rewrite the overall dynamics 
as 
\begin{equation}\label{eq:s1_dyna_general}
    \dot{z}=Az+f(t).
\end{equation}
In order to write~\eqref{eq:s1_dyna_general} in the form of~\eqref{eq:dyna}, we perform the change of variables as $x=z-g(t)$, where $g(t)$ satisfies the differential equation $\dot{g}(t) = A g(t) + f(t)$, $g(0) = 0$, which can be calculated in a distributed way and treated as a reference signal for $z$. Then,~\eqref{eq:s1_dyna_general} is equivalent to $\dot{x}=Ax$. It is well known that it is possible to choose the set $\mathcal{V}_g$ and design limited-agent interactions to ensure that $A$ is a Hurwitz matrix~\cite{hu2010distributed}. 
Hence, $z$ will track the reference signal $g(t)$. A group of agents may utilize this approach to adjust the relative formation and move safely through a narrow corridor. However, a disruption in agent interactions may result into collision. Please refer to Section~\ref{sec:formation-control-example} for more information.

\subsubsection{Power systems}
The dynamics of generator buses in power networks are governed by the swing equation $M \ddot{\theta} + D \dot{\theta} + L \theta = u$, which is obtained after the linearization of the nonlinear power flow equations around a stationary operating point~\cite{khajenejad2025distributed,pasqualetti2013attack} 
and $L$ is the generator network Laplacian obtained after a Kron reduction~\cite{dorfler2014sparsity} of the load buses in the network admittance matrix. By neglecting the generator inertia dynamics, the inertia term can be omitted by setting $M = 0$, which yields a first-order network dynamics model~\cite{wu2014sparsity} that retains the essential interconnection topology, $\dot{\theta} = - D^{-1} L  \theta + D^{-1} u.$ A feedback gain matrix $K$ can be designed such that the closed-loop dynamics become $\dot{\theta} = - D^{-1} (L + K)  \theta,$ where the matrix $K$ represent either local or wide-area control interactions among the nodes~\cite{dorfler2014sparsity}. By properly designing $K$, the closed-loop matrix $A = - D^{-1} (L + K)$ can be made Hurwitz while remaining sparse. The Laplacian $L$ is fixed and characterizes the physical topology of the power network, whereas different designs of $K$ correspond to different feedback strategies. Switching between such feedback gains $K_1, K_2$ enables changes in the control interconnection structure and the overall resilience can be enhanced. We refer the reader to Section~\ref{subsec:power-network-example}, where we provide details of such an example and evaluate its resilience.

\section{Optimal switching for commutative networks}
In this section, we focus on Problem 1. We first propose a definition of $\alpha(S(t))$ via Floquet theory. 
Then, we present an algorithm to 
solve Problem~1.

\subsection{Resilience metric for switching systems}
As $S(t)$ is piecewise continuous, bounded and periodic; according to Floquet theory;  cf.~Section~\ref{sub:Floquet}, the time-averaged trajectories of~\eqref{eq:switch} can be captured by the solution of an associated system counterpart.  Thus, our quantification of $\alpha(S(t))$ will rely on the linear time-invariant system
\begin{equation} \label{eq:LTI}
        \dot{z}(t)=Qz(t).
\end{equation}
From~\eqref{eq:R_pre}, we can compute $R$ as
\begin{equation}\nonumber
\begin{aligned}
     R=\Phi(T,0):&=\prod_{i=1}^m e^{B\Delta t_{i,2}}e^{A\Delta t_{i,1}}\\&=e^{B\Delta t_{m,2}}e^{A\Delta t_{m,1}}\cdots e^{B\Delta t_{1,2}}e^{A\Delta t_{1,1}},
\end{aligned}     
\end{equation}
which implies that $Q$ in~\eqref{eq:Q_pre} can be obtained as
\begin{equation}\label{eq:Q}
    Q=\frac{1}{T}\ln{R}=\frac{1}{T}\ln(\prod_{i=1}^m e^{B\Delta t_{i,2}}e^{A\Delta t_{i,1}}).
\end{equation}

It is known that the switched system~\eqref{eq:switch} is exponentially stable if and only if $Q$ in~\eqref{eq:LTI} is Hurwitz~\cite{gokccek2004stability}. 
This naturally leads to the definition 
\begin{equation}\label{eq:metric}
  \alpha(S(t)):=\alpha(Q)=\max\limits_{\lambda\in\Lambda(Q)} \Real\lambda,
\end{equation}
where $\Lambda(Q)$ is the spectrum of $Q$. We further investigate~\eqref{eq:Q}  to understand how $\alpha(Q)$ relates to $\alpha(A)$, $\alpha(B)$. The  expression in~\eqref{eq:Q} can be further simplified by assuming that $A$ and $B$ commute, as stated in the following assumption. 
\begin{assumption}[Commutativity] \label{ass1}
The two networks represented by matrices $A$ and $B$ commute; i.e.~$AB=BA$. \oprocend 
\end{assumption}
\begin{assumption}[Network stability] \label{ass1_1}
The two networks represented by matrices $A, B\in \mathbb{R}^{n\times n}$ are Hurwitz; i.e., $\alpha(A), \alpha(B)<0$. \oprocend 
\end{assumption}

Leveraging Assumption~\ref{ass1}, $Q$ can be simplified to
\begin{equation}\nonumber
    Q=\frac{1}{T}\ln(e^{B\Delta t_2+A\Delta t_1}),
\end{equation}
where $\Delta t_1=\sum_{i=1}^m \Delta t_{i,1}$, $\Delta t_2=\sum_{i=1}^m \Delta t_{i,2}$ and $\Delta t_1+\Delta t_2=T$. Let $k$ denote the time ratio during which $A$ is active within one period, i.e., $k=\Delta t_1/T, 0\leq k \leq 1$. Then, $B$ is active within the ratio $1-k$. We can rewrite $Q$ as 
\begin{equation}\label{eq:Q_com2}
    Q=\frac{1}{T}({B\Delta t_2+A\Delta t_1})=(1-k)B+kA.
\end{equation}
Therefore, we can analyze~\eqref{eq:metric} based on the network properties of $A, B$ and the dwell-time decision variable $k$. 

\subsection{Time-varying actuation to enhance resilience}

The optimization problem~\eqref{eq:optimiation_ori} can be now reformulated as the following resilience-optimal switching $(\ROS)$ problem: 
\begin{equation} \nonumber
(\ROS)\hspace{5mm}
\begin{aligned}
\min_{k} \quad & \alpha(Q)=\alpha(kA+(1-k)B)\\
\textrm{s.t.} \quad & 0\leq k \leq 1.\\
\end{aligned}
\end{equation}
To solve it, we exploit the fact that commutative matrices are simultaneously upper triangularizable~\cite{horn2012matrix}. Therefore, 
\begin{equation}\label{eq:upptri}
    A=P^{-1}U_AP, \quad B=P^{-1}U_BP, 
\end{equation}
where $U_A$ and $ U_B$ are upper-triangular matrices with the diagonal entries given by the eigenvalues of $A$ and $B$, respectively, and $P$ is formed by the common eigenvectors to both matrices. Note that, if both $A$ and $B$ have distinct eigenvalues, $U_A$ and $U_B$ will be diagonal matrices.  In what follows, we group the eigenvalues of $A$ and $B$ in pairs, $(\lambda_i, \mu_i)$, so that $\lambda_i$, (resp.~$\mu_i$) is an eigenvalue of $A$ (resp.~of $B$), and both share the same eigenvector. Without loss of generality, we assume that $(\lambda_1,\mu_1)$ (resp.~$(\lambda_2,\mu_2)$) is the eigenvalue pair such that $\alpha(A) = \Real\lambda_1 \equiv \alpha_A$ (resp.~$\alpha(B) = \Real\mu_2 \equiv \alpha_B$). In this way, \eqref{eq:Q_com2}~becomes
\begin{equation}\nonumber
    Q=kA+(1-k)B=P^{-1}(kU_A+(1-k)U_B)P.
\end{equation}
Therefore, the spectrum of $Q$ is the same as the spectrum of $kU_A+(1-k)U_B$, which, in turn, it  also is an upper triangular matrix. More precisely, 
\begin{equation}\nonumber
    \Lambda(Q)=\{k\lambda_i+(1-k)\mu_i,|\,\lambda_i \in \Lambda(A), \mu_i \in \Lambda(B)\}.
\end{equation}
With this, $\ROS$ is equivalent to the problem of finding an optimal switching ratio  $k$ such that the real part of the rightmost eigenvalue of $Q$ is minimized; or, equivalently, so that the
so-called resilience-optimal switching $\rROS$ problem
\begin{equation}\nonumber
(\rROS)\hspace{2mm}
\begin{aligned}
\min_{k} \quad & x\\
\textrm{s.t.} \quad & 0\leq k \leq 1,\\
  &x\geq k\Real\lambda_i+(1-k) \Real\mu_i, \; i\in \until{n}, \\
\end{aligned}
\end{equation}
is solved. Denote by $k^\star$ the optimal $k$ that solves $\rROS$, and by $x^\star = \alpha^\star(Q)$  the optimal value that solves $\rROS$.

Next, we look for  conditions that can ensure 
$\alpha(S(t))<\min(\alpha(A),\alpha(B))$ while providing a bound for $\alpha^\star(Q)$, under the following assumption. 

\begin{assumption}[Unique Max]\label{ass2}
There is a unique eigenvalue of $A$ (resp.~$B$) that attains  $\alpha(A)$ (resp.~$\alpha(B)$). 
That is, $|\{i \mid \Real\lambda_i = \alpha(A)\}| = |\{j \mid \Real\mu_j = \alpha(B)\}| = 1$. \oprocend
\end{assumption}

\begin{theorem}[Bound on switched system resilience]\label{thm:1} 

Let $A$ and $B$ be two networks satisfying the commutativity Assumption~\ref{ass1}, and which define the corresponding switched system~\eqref{eq:switch}. 
If Assumption~\ref{ass2} on the unique max holds, then the resilience of the switched system, characterized by~\eqref{eq:metric}, is enhanced; i.e.~$\alpha^\star(Q) < \min (\alpha_A, \alpha_B)$, if and only if 
\begin{equation}\label{eq:iff_con}
    \beta_B < \alpha_A \text{ and } \beta_A < \alpha_B,
\end{equation}
where $(\alpha_A, \beta_B) = (\Real\lambda_1, \Real\mu_1)$ and $(\beta_A, \alpha_B) = ( \Real\lambda_2, \Real\mu_2)$, respectively. Furthermore, $\alpha^\star(Q)$ is bounded by the following inequality: 
\begin{equation}\label{eq:thm_bound}
\begin{aligned}
    &\max (\beta_A, \beta_B)<\alpha_A+\frac{\delta_A}{\delta_A+\delta_B}(\beta_B-\alpha_A)\\ &=\alpha_B+\frac{\delta_B}{\delta_A+\delta_B}(\beta_A-\alpha_B)\leq \alpha^\star(Q)< \min (\alpha_A,\alpha_B),
\end{aligned}
\end{equation}    
where $\delta_A=\alpha_A-\beta_A$ and $\delta_B=\alpha_B-\beta_B$.
\end{theorem}
\begin{proof}
We first prove the ``only if" part (the necessary condition); that is, if $\alpha^\star(Q) < \min(\alpha_A, \alpha_B)$, then condition~\eqref{eq:iff_con} must hold. According to $\rROS$, we want to minimize $x$ while satisfying the constraints $x\geq k\Real \lambda_i+(1-k) \Real\mu_i,$ for all $i\in \until{n}.$ 
Let us define $f_i(k)=k\Real \lambda_i+(1-k) \Real\mu_i$, $i \in \until{n}$. Without loss of generality, denote by $f_1(k)=k\alpha_A+(1-k) \beta_B$ and $f_2(k)=k\beta_A+(1-k) \alpha_B$, which we refer to as ``dominant segments".\footnote{We refer to $f_1,f_2$ as dominant segments because there are values of $k$ for which they dominate all others: for $k = 1$, we have $f_1(1) = \alpha_A \geq f_i(1)$ for all~$i$, and for $k = 0$, $f_2(0) = \alpha_B \geq f_i(0)$, for all~other $i$.}
Note that, if $k^\star=0$, the optimal value is $x^\star=\alpha_B$, which means that we fix the network to $B$. While when $k^\star=1$, the optimal value is  $x^\star=\alpha_A$, setting the network equal to $A$. 
Both cases are trivial and the resilience does not improve with respect to that provided by either $A$ or $B$. The only possible way to improve resilience; that is, $\alpha^\star(Q) < \min \{\alpha_A, \alpha_B\}$, requires that $k^\star\in (0,1)$. To study this case, we construct the following optimization problem: 
    \begin{equation}\nonumber
(\AOone)\hspace{2mm}
\begin{aligned}
\min_{k} \quad & x\\
\textrm{s.t.} \quad & 0\leq k \leq 1,\\
  &x\geq k\alpha_A+(1-k) \beta_B, \\
  &x\geq k\beta_A+(1-k) \alpha_B,   \\
\end{aligned}
\end{equation}
and let $\underline{x}^\star$ denote its optimal solution (similarly, let $\underline{k}^\star$ be the corresponding optimal $k$ for this problem). 
We know that $\underline{x}^\star \leq x^\star$, since the constraints of $\AOone$ form a subset of those of $\rROS$. The KKT conditions for $\AOone$ characterize its solutions and are obtained next. First, define
\begin{equation}\nonumber
\begin{aligned}
    &\mathcal{L}(k, x, \theta_1, \theta_2, \nu_1, \nu_2) = x + \theta_1 \left( k \alpha_A + (1-k) \beta_B - x \right) \\&+ \theta_2 \left( k \beta_A + (1-k) \alpha_B - x \right) - \nu_1 k + \nu_2 (k - 1),
\end{aligned}    
\end{equation}
where $\theta_1, \theta_2 \geq 0$ are the Lagrange multipliers associated with $x\geq f_1(k)$ and $x\geq f_2(k)$; respectively, $ \nu_1, \nu_2 \geq 0$ are the Lagrange multipliers associated with $0\leq k\leq 1$. The KKT conditions for the optimal solutions are as follows: 
\begin{itemize}
    \item Stationarity: First, $ \frac{\partial \mathcal{L}}{\partial x} = 1 - \underline{\theta}_1^\star - \underline{\theta}_2^\star = 0,$ which leads to $\underline{\theta}_1^\star + \underline{\theta}_2^\star = 1$. Second, $\frac{\partial \mathcal{L}}{\partial k} = \underline{\theta}_1^\star \left( \alpha_A - \beta_B \right) + \underline{\theta}_2^\star \left( \beta_A - \alpha_B \right) - \underline{\nu}_1^\star + \underline{\nu}_2^\star = 0.$
    \item Primal feasibility: 
    \begin{equation} \nonumber
         \begin{aligned}
   &\underline{x}^\star \geq \underline{k}^\star \alpha_A + (1-\underline{k}^\star) \beta_B, \\
   &\underline{x}^\star \geq \underline{k}^\star \beta_A + (1-\underline{k}^\star) \alpha_B, \\
   & 0 \leq \underline{k}^\star \leq 1.
   \end{aligned}
    \end{equation}
    \item Dual feasibility: $\underline{\theta}_1^\star, \underline{\theta}_2^\star, \underline{\nu}_1^\star, \underline{\nu}_2^\star \geq 0.$
    \item Complementary slackness: 
    \begin{equation} \nonumber
         \begin{aligned}
   &\underline{\theta}_1^\star ( \underline{k}^\star \alpha_A + (1-\underline{k}^\star) \beta_B - \underline{x}^\star) = 0, \\
   &\underline{\theta}_2^\star ( \underline{k}^\star \beta_A + (1-\underline{k}^\star) \alpha_B - \underline{x}^\star) = 0, \\
   & \underline{\nu}_1^\star \cdot (-\underline{k}^\star) = 0,\\
   & \underline{\nu}_2^\star \cdot (\underline{k}^\star - 1) = 0.
   \end{aligned}
    \end{equation}
    
\end{itemize}
Note that, if  $x^\star = \alpha^\star(Q)< \min (\alpha_A,\alpha_B)$, then $\underline{x}^\star  < \min (\alpha_A,\alpha_B)$. Therefore,  it must be  that $0< \underline{k}^\star < 1$. This, together with the complementary slackness, implies that $\underline{\nu}_1^\star=\underline{\nu}_2^\star=0$, which simplifies the stationarity conditions as
         \begin{align}
   &\underline{\theta}_1^\star+\underline{\theta}_2^\star=1, \label{eq:KKT_equality1}\\
   & \underline{\theta}_1^\star \left( \alpha_A - \beta_B \right) + \underline{\theta}_2^\star \left( \beta_A - \alpha_B \right)=0.\label{eq:KKT_equality3}
   \end{align}
If $\underline{\theta}_1^\star=0$ or $\underline{\theta}_2^\star=0$ (or equivalently, from~\eqref{eq:KKT_equality1}, if $\underline{\theta}_2^\star=1$ or 
$\underline{\theta}_1^\star=1$,  respectively), then, from~\eqref{eq:KKT_equality3}, we obtain $\beta_A = \alpha_B$ or $\alpha_A = \beta_B$; respectively. From the first two equations of Complementary slackness, this implies that $\underline{x}^\star$ will either be $\underline{x}^\star = \alpha_B$ or $\underline{x}^\star = \alpha_A$; respectively, and resilience is not improved. Therefore, $\underline{\theta}_1^\star,\underline{\theta}_2^\star> 0$ must hold, which confirms that the dominant-segment constraints are active, 
\begin{equation}\label{eq:eqaility_k}
    \underline{x}^\star = \underline{k}^\star \alpha_A + (1-\underline{k}^\star) \beta_B=\underline{k}^\star \beta_A + (1-\underline{k}^\star) \alpha_B.
\end{equation}

Assume now that $\underline{\theta}_1^\star, \underline{\theta}_2^\star > 0$ and consider~\eqref{eq:KKT_equality3}. The latter holds in two possible cases. First, suppose that $\alpha_A - \beta_B = 0$ and $\beta_A - \alpha_B = 0$. From~\eqref{eq:eqaility_k}, it must be that $\alpha_A = \beta_B = \beta_A = \alpha_B$. This leads to a contradiction with Assumption~\ref{ass2}, on the Unique Max. The second case is $\frac{\alpha_A - \beta_B}{\beta_A - \alpha_B} = -\frac{\underline{\theta}_2^\star}{\underline{\theta}_1^\star} < 0$. 
Due to the definition of the spectral abscissa~\eqref{eq:spec_abs} and Assumption~\ref{ass2}, the conditions $\beta_B > \alpha_A$ and $\beta_A > \alpha_B$ cannot both be true, as it leads to  $\beta_A > \alpha_A$, a contradiction. Consequently, $(\alpha_A - \beta_B)(\beta_A - \alpha_B)^{-1} < 0$ only when~\eqref{eq:iff_con} holds.

The ``if" part (the sufficient condition) follows from geometric considerations. Since $\alpha_A > \beta_A$ and $\alpha_B > \beta_B$, then $f_1(k)$ (resp.~$f_2(k)$) is increasing (resp.~decreasing) over $[0,1]$. 
Note that, under condition~\eqref{eq:iff_con} and Assumption~\ref{ass2}, the solution to $\AOone$, $\underline{k}^\star \in (0,1)$ and $\underline{x}^\star$ must be attained at the value of intersection, $\underline{x}^\star = f_1(\underline{k}^\star) = f_2(\underline{k}^\star)$ (otherwise, $\underline{x}^\star$ is achieved at $\underline{k}^\star = 0$ or $\underline{k}^\star = 1$, which implies $\alpha_A = \beta_A$, or $\alpha_B = \beta_B$, contradicting Assumption~\ref{ass2}). We then investigate the case when $x\geq f_i(k)$ is active, for some $i\in \{3,\dots, n\}$.  Consider the segment $\overline{f}(k) = \alpha_B + k (\alpha_A - \alpha_B)$, $k \in [0,1]$. Note that it always holds that $\overline{f}(k) > f_i(l)$, for all $l,k \in [0,1]$, and for any $f_i(k) = \Real(\mu_i) + k (\Real(\lambda_i) - \Real(\mu_i))$, $i \in \{3,\dots,n\}$. Otherwise, and due to Assumption~\ref{ass2}, both segments intersect at a $\overline{k} \in (0,1)$, and either $f_i$ dominates over $\overline{f}$ on $(\overline{k},1]$, or on $[0,\overline{k})$. But if the former, then  $\alpha_A = \overline{f}(1) < f_i(1) = \Real(\lambda_i)$, and, if the latter, then $\alpha_B = \overline{f}(0) < f_i(0) = \Real(\mu_i)$, a contradiction. Therefore, if a solution $k^\star$ is attained at an $f_i$, $i \in \{3,\dots,n\}$, we have that either $x^\star = f_i(k^\star) < \alpha_A = \overline{f}(1)$, $ x^\star = f_i(k^\star) < \alpha_B=\overline{f}(0)$. In other words, $x^\star = \alpha^\star(Q) < \min \{\alpha_A,\alpha_B\}$.

To prove the inequality~\eqref{eq:thm_bound}, we already have that $\alpha^\star(Q) < \min(\alpha_A, \alpha_B)$ provided that condition~\eqref{eq:iff_con} is satisfied. We then solve~\eqref{eq:eqaility_k}, and obtain $\underline{x}^\star$ as $\underline{x}^\star=\alpha_A+\frac{\delta_A}{\delta_A+\delta_B}(\beta_B-\alpha_A)=\alpha_B+\frac{\delta_B}{\delta_A+\delta_B}(\beta_A-\alpha_B)$ with $\delta_A=\alpha_A-\beta_A$ and $\delta_B=\alpha_B-\beta_B$, and the corresponding $k=\frac{\delta_B}{\delta_A+\delta_B}\in (0,1)$. 
Given that $\underline{x}^\star \le \alpha^\star(Q) $, the equality and second inequality from the left of~\eqref{eq:thm_bound} hold. Furthermore, $\alpha_A+\frac{\delta_A}{\delta_A+\delta_B}(\beta_B-\alpha_A)-\beta_B=k(\alpha_A-\beta_B)>0$ and $\alpha_B+\frac{\delta_B}{\delta_A+\delta_B}(\beta_A-\alpha_B)-\beta_A=(1-k)(\alpha_B-\beta_A)>0$. 
This establishes the satisfaction of~\eqref{eq:thm_bound}. 
\end{proof}

We generalize Theorem~\ref{thm:1} to the case of multiple eigenvalues, and the results are summarized in Corollary~\ref{coro:1}. We first define the index set $\mathcal{J}$ as $\mathcal{J}=\{j\mid (\lambda_j,\mu_j) \mbox{ with }\Real \lambda_j=\alpha_A \mbox{ or }\Real \mu_j=\alpha_B\}$ i.e., the pairs of eigenvalues for which either $\lambda_j$ or $\mu_j$ attain the spectral abscissa value. The dominant segments are defined as $f_j(k)=k\Real \lambda_j+(1-k)\Real \mu_j, j\in \mathcal{J}$, and another auxiliary optimization problem is proposed:
    \begin{equation}\nonumber
(\AOtwo)\hspace{2mm}
\begin{aligned}
\min_{k} \quad & x\\
\textrm{s.t.} \quad & 0\leq k \leq 1,\\
  &x\geq f_j(k), j\in \mathcal{J}.
\end{aligned}
\end{equation}

\begin{corollary}[Bound on switched system resilience, general case]\label{coro:1} 
Let the prerequisites of Theorem~\ref{thm:1} hold, except for Assumption~\ref{ass2}. Then, $\alpha^\star(Q) < \min (\alpha_A, \alpha_B)$ if and only if the following conditions hold:
\begin{equation}\label{eq:iff_con_2}
\begin{aligned}
    &\mbox{if }\Real \lambda_i=\alpha_A \mbox{ for some } i, \mbox{then }\Real\mu_i<\alpha_A;
\end{aligned}    
\end{equation}
\begin{equation}\label{eq:iff_con_3}
\begin{aligned}
    &\mbox{if }\Real \mu_j=\alpha_B \mbox{ for some } j, \mbox{then }\Real\lambda_j<\alpha_B.
\end{aligned}    
\end{equation}
Furthermore, $\alpha^\star(Q)$ is bounded by the following inequality:
\begin{equation}\label{eq:thm_bound_coro}
    \underline{x}^{\star\star}\leq \alpha^\star(Q)< \min (\alpha_A,\alpha_B),
\end{equation}      
where $\underline{x}^{\star\star}$ is the optimal value of problem $\AOtwo$.
\end{corollary}
\begin{proof}
    The proof extends the arguments of Theorem~\ref{thm:1}. In cases where multiple eigenvalues achieve the spectral abscissa value, every pair must satisfy a similar if and only if condition as stated in~\eqref{eq:iff_con}, leading to the conditions~\eqref{eq:iff_con_2} and~\eqref{eq:iff_con_3}. Moreover, since there are more dominant segments serving as constraints in $\AOtwo$ compared to $\AOone$, the resulting lower bound is obtained by solving $\AOtwo$. Consequently, the bound in~\eqref{eq:thm_bound_coro} is a generalization of~\eqref{eq:thm_bound}.   
\end{proof}
\begin{remark}
[Enhancing resilience] {\rm To augment resilience through time-varying actuation,  the networks $A$ and $B$ need to exhibit opposite properties that lead to a “mismatch” in their eigenvalues, as indicated by~\eqref{eq:iff_con},~\eqref{eq:iff_con_2},~\eqref{eq:iff_con_3}. Furthermore, the bounds~\eqref{eq:thm_bound} and~\eqref{eq:thm_bound_coro} are linked to the eigenvalue pairs that determine the resilience of each individual network. }\oprocend
\end{remark}

Finally, we summarize in Algorithm~\ref{algorithm1} the proposed approach for deriving the optimal switching between $A$ and $B$. 

\begin{algorithm}[!ht]
\PrintSemicolon
  \KwInput{$A, B$ s.t. $AB=BA$}
  \KwOutput{$k^\star,\alpha^\star(Q)$}

  Upper triangularize $A, B$ as~\eqref{eq:upptri}\\
  Pair the eigenvalues $(\lambda_i, \mu_i)$ of $A, B$  with respect to the same eigenvector\\
  Compute the real part of $\lambda_i, \mu_i$ \\
  Solve the linear programming $\rROS$\\
  \Return  $k^\star,\alpha^\star(Q)$
\caption{$\boldsymbol{\optswitch}$ --- Computation of \textbf{Opt}imal \textbf{Switch}ing}
\label{algorithm1}
\end{algorithm}


\section{Network structural optimization under homogeneous sparsity constraints}
In this section, we focus on  Problem~2 to optimize resilience from a topological-design perspective. In this way, given a network $A$, we aim to find a resilience-enhancing complementary network, $B$, which we can use together with an optimal switching strategy as in~\eqref{eq:switch} so that $\alpha(S(t))<\min(\alpha(A),\alpha(B))$. In doing so, we will consider various optimization formulations. 

Our first problem, $\lzeroRec$, is the following:


\begin{equation}\nonumber
(\lzeroRec)\hspace{5mm}
\begin{aligned}
\min_{k, B} \quad & \alpha(kA+(1-k)B)+ {\|\Gamma \odot B\|}_{\ell_0}\\
\textrm{s.t.} \quad & AB=BA, \\
& 0\leq k \leq 1, \\
& \Tr(B)=\Tr(A). 
\end{aligned}
\end{equation}
In $\lzeroRec$, we optimize over $B$ commutative with $A$, as well as over the switching  between $A$ and $B$. The objective function consists of two parts: The first part is as in~$\ROS$ (or~$\rROS$), to optimize the topology-switching ratio, while the second part involves $ {\|\cdot\|}_{\ell_0}$,
a sparsity-promoting penalty function defined on the Hadamard product of $B$ with a weighting matrix $\Gamma:=[\gamma_{ij}]\in \mathbb{R}^{n\times n}$.
The constraint $AB=BA$ ensures commutativity, while the second  $0\leq k \leq 1 $ characterizes  the activation of $A$ within one period. The third constraint $\Tr(B)=\Tr(A)$ is introduced 
to ensure that 
$B$ maintains a spectral scale that is comparable to that of $A$.

The $\ell_0$-norm is commonly used to characterize sparsity. 
Unfortunately, this leads to optimization problems that are nonconvex and NP-hard. 
A common approach to address this issue is to replace the $\ell_0$-term ${\|\cdot\|}_{\ell_0}$ with an $\ell_1$-term, ${\|\cdot\|}_{\ell_1}$. 
The $\ell_1$ relaxation results into 
\begin{equation}\nonumber
(\loneRec)\hspace{5mm}
\begin{aligned}
\min_{k, B} \quad & \alpha(kA+(1-k)B)+ {\|\Gamma \odot B\|}_{\ell_1}\\
\textrm{s.t.} \quad & AB=BA, \\
& 0\leq k \leq 1, \\
& \Tr(B)=\Tr(A). 
\end{aligned}
\end{equation}

Problems $\loneRec$ and $\lzeroRec$ are in general not equivalent, being the first a relaxation of the second.

In what follows, we characterize the algebraic conditions on $B$ that ensure it commutes with $A$. Building upon these properties, we can rewrite $\loneRec$ and obtain new insight on its relationship with $\lzeroRec$. For a given $A$, $\mathcal{C}(A)=\{B\mid AB=BA\}$ denotes the \textit{centralizer} of $A$~\cite{horn2012matrix}. Let $\lambda_{A}=[\lambda_1,\cdots,\lambda_n]^T$ denote the vector of all eigenvalues of $A$. The Vandermonde matrix $\Lambda_p$ formed by $\lambda_{A}$ is defined as 
\begin{equation}\label{eq:vandermonde}
       \Lambda_p:= \begin{bmatrix}
1 & \lambda_1 & \lambda_1^2 & \cdots & \lambda_1^{n-1} \\ 
1 & \lambda_2 & \lambda_2^2 & \cdots & \lambda_2^{n-1} \\
\vdots & \vdots & \vdots & \ddots &\vdots \\
1 & \lambda_n & \lambda_n^2 & \cdots & \lambda_n^{n-1} 
\end{bmatrix}.
    \end{equation}
The following lemma states the algebraic properties of $\mathcal{C}(A)$ using the powers of $A$, which follows from Theorem~3.2.4.2 of~\cite{horn2012matrix}. Here, we provide an alternative proof to this result via the Vandermonde matrix, which will be helpful in the sequel.  
\begin{lemma}[Algebraic characterization of a matrix centralizer]\label{lemma:1} 
Suppose that $A\in \mathbb{R}^{n\times n}$ has distinct eigenvalues. Then $B\in \mathbb{R}^{n\times n}$ is commutative with $A$ if and only if $B\in \mathcal{C}(A)=\spa \{I, A, A^2,\cdots, A^{n-1}\}$. 
\end{lemma}
\begin{proof}
The ``if'' part follows immediately: we first write $B$ as a linear combination of the basis matrices, i.e., $B=\sum_{i=0}^{n-1}c_iA^i=c_0I+c_1A\dots+c_{n-1}A^{n-1}$. We then can verify that $AB=BA=\sum_{i=0}^{n-1}c_iA^{i+1}$, thus $B$ is commutative with $A$. The ``only if'' part can be proved via Theorem~3.2.4.2 of~\cite{horn2012matrix} directly.
Instead, we  prove it by leveraging matrix properties to establish the connection between the $A, B$ spectra. Since $A, B$ are commutative, they are simultaneously upper-triangularizable as $A=P^{-1}XP, B=P^{-1}YP$, where $X$ and $ Y$ are the upper-triangular matrices with diagonal entries given by the eigenvalues of $A$ and $B$, respectively. Furthermore, $X $ is diagonal, $X = \diag{\lambda_1,\cdots, \lambda_n}$, with the diagonal containing the distinct eigenvalues of $A$. Since $AB=BA$, then $XY=YX$, which yields $\lambda_iy_{ij}=\lambda_jy_{ij}$ where $Y = (y_{ij})$. Due to the fact that $A$ has distinct eigenvalues, $(\lambda_i - \lambda_j) y_{ij}  =0$ implies $y_{ij}=0$ for $i\neq j$. Therefore, $Y = \diag{\mu_1,\cdots, \mu_n}$ is also diagonal with its eigenvalues in the diagonal. Note that the statement we aim to prove is invariant under similarity transformations $P$, and thus, w.l.o.g.~we can focus on diagonal matrices. We have that $\sum_{i=0}^{n-1}c_iX^i=\diag{\sum_{i=0}^{n-1}c_i\lambda_1^i,\cdots, \sum_{i=0}^{n-1}c_i\lambda_n^i}$. 
Therefore, we obtain $\Lambda_pc=\mu_B$, where $\Lambda_p$ is the Vandermonde matrix as in~\eqref{eq:vandermonde},  $c=[c_0, c_1, \cdots, c_{n-1}]^T$, and $\mu_{B}=[\mu_1, \mu_2,\cdots,\mu_n]^T$.
The Vandermonde matrix $\Lambda_p$ is invertible since $X$ has distinct eigenvalues~\cite{horn2012matrix}, thus we can always obtain a coefficient vector $c$ that is a solution to the previous equation. Pre and post-multiplying $\Lambda_p c = \mu_B$ by $P^{-1}, P$, respectively, implies that $B= \sum_{i=0}^{n-1} c_i A^i, c_i \in \mathbb{C}$.  This allows us to show that $c \in \mathbb{R}^n$. Specifically, let $c_i = a_i + b_i \mathrm{i}$ for $i \in \untilzn{n-1}$, where $\mathrm{i}$ is the imaginary unit, and $a_i, b_i \in \mathbb{R}$. Then, $B = \left( \sum_{i=0}^{n-1} a_i A^i \right) + \mathrm{i} \left( \sum_{i=0}^{n-1} b_i A^i \right)$. Since $A, B \in \mathbb{R}^{n \times n}$ and the set $\{I, A, A^2, \ldots, A^{n-1}\}$ is linearly independent, it follows that $b_i = 0$ for all $i$. This completes the proof. 
\end{proof}

We now establish the equivalence between $\lzeroRec$, and $\loneRec$, under certain choice of weighting parameters. 

\begin{theorem}[Optima of $\lzeroRec$ and $\loneRec$]\label{thm:2} Suppose $\gamma_{ij}$ of the sparsity-promoting weighting matrix, $\Gamma$, are all equal; that is, $\gamma_{ij} = \gamma$, which we refer to as a homogeneous sparsity constraint. If the optimal $B^\star$ for $\lzeroRec$ (resp. $\loneRec$) satisfies Assumption~\ref{ass1_1}, on network stability, then $B^\star=\frac{\Tr(A)}{n}I$ is an optimal solution for both $\lzeroRec$ and $\loneRec$.
\end{theorem} 
\begin{proof}
  To show the result, we first prove that $B^\star=\frac{\Tr(A)}{n}I$ is optimal for the following optimization: 
    \begin{equation}\nonumber
(\AOthr)\hspace{5mm}
\begin{aligned}
\min_{k, B} \quad & \alpha(kA+(1-k)B)\\
\textrm{s.t.} \quad & AB=BA, \\
& 0\leq k \leq 1, \\
& \Tr(B)=\Tr(A). 
\end{aligned}
\end{equation}
The optimal solution for $\AOthr$ is attained at $k^\star=0$, and $B^\star=\frac{\Tr(A)}{n}I$. To see this, first note that $B^\star\in \mathcal{C}(A)$, thus $BA = AB$. The third constraint also holds as $\Tr(B^\star)=\Tr(\frac{\Tr(A)}{n}I)=n\frac{\Tr(A)}{n}=\Tr(A)$. The minimum value of the objective function is given by $\alpha^\star(Q)=\frac{\Tr(A)}{n}$, 
which can be shown by contradiction. Suppose that $\alpha^\star(Q)<\frac{\Tr(A)}{n}$. Then, the real parts of the eigenvalues of $Q$ are smaller than $\frac{\Tr(A)}{n}$ due to the definition of spectral abscissa, which implies that $\Tr(Q)<n\frac{\Tr(A)}{n}=\Tr(A)$. However, it must be that $\Tr(Q)=\Tr(kA+(1-k)B)=k\Tr(A)+(1-k)\Tr(B)=\Tr(A)=\Tr(B)$, a contradiction. Thus, $\alpha^\star(Q)=\frac{\Tr(A)}{n}$ is optimal with $B^\star=\frac{\Tr(A)}{n}I$ and $k^\star = 0$. 

We then prove that $B^\star=\frac{\Tr(A)}{n}I$ is also optimal for both  
$\lzeroRec$ and $\loneRec$ when considering the sparsity-promoting cost part. For $\lzeroRec$, we show that ${\|\Gamma \odot B\|}_{\ell_0}$ is also minimized with $B^\star=\frac{\Tr(A)}{n}I$. First, ${\|\Gamma \odot B^\star\|}_{\ell_0}=n\gamma$ for the chosen $B^\star$. Due to the assumption that the optimal solution $B^\star$ is Hurwitz, we prove optimality by contradiction.
Suppose there exists $B$ such that $ {\|\Gamma \odot B\|}_{\ell_0}<n\gamma$. This indicates that $B$ has fewer than $n$ nonzero entries, thus $B$ is singular, which leads to a contradiction. Therefore, $B^\star=\frac{\Tr(A)}{n}I$ is optimal for both the resilience part and the sparsity-promoting part, which means that it is optimal for $\lzeroRec$.   

For $\loneRec$, we also show that the objective function $ {\|\Gamma \odot B\|}_{\ell_1}$ on sparsity is minimized with $B^\star=\frac{\Tr(A)}{n}I$. We can compute that $\|\Gamma \odot B^\star\|_{\ell_1}=-\gamma\Tr(A)$. 
According to the third constraint of $\loneRec$, we have that $ {\|\Gamma \odot B\|}_{\ell_1}\geq \gamma\sum_i |b_{ii}|\geq -\gamma\sum_i b_{ii} = -\gamma\Tr(B)=-\gamma\Tr(A)$. 
Therefore, $B^\star=\frac{\Tr(A)}{n}I$ is optimal for both the resilience part and the sparsity-promoting part of $\loneRec$, which means that it is optimal for problem $\loneRec$. Hence, we conclude that, when $\gamma_{ij}=\gamma, \forall i, j$ in $\lzeroRec$ and $\loneRec$, both problems are equivalent with a solution $B^\star=\frac{\Tr(A)}{n}I$. 
\end{proof}
Theorem~\ref{thm:2} provides a sufficient condition under which problems $\lzeroRec$ and the $\loneRec$ are equivalent. However, the optimal solution $B^\star=\frac{\Tr(A)}{n}I$ places all non-zero entries of $B$ on its diagonal due to the uniform choice of $\gamma$. In practice, we may be interested in alternative solutions guided by additional design constraints, e.g., restricting the support of $B$ to a subset $\mathbb{S}$ of edges that an operator can modify. This solution should also be harder to guess by an adversary.
Motivated by this, we study the convexity of an associated problem and explore how to select the weighting parameters to avoid trivial solutions. 

\section{Network structural optimization under heterogeneous sparsity constraints} 
Based on Lemma~\ref{lemma:1}, we can write $B$ as a linear combination $B=c_0I+c_1A+\dots+c_{n-1}A^{n-1}$ or in vectorized form as 
\begin{equation}\nonumber
    \ve(B)=c_0\ve(I)+c_1\ve(A)+\dots+c_{n-1}\ve(A^{n-1})\in \mathbb{R}^{n^2}.
\end{equation}
Denote $A_p=[\ve(I),\cdots,\ve(A^{n-1})]\in \mathbb{R}^{n^2\times n}$, where we form $A_p$ by stacking the vectorized basis matrices as columns, and $c=[c_0\cdots c_{n-1}]^T$. Then we have that 
\begin{equation}\label{eq:vec_A2} 
    \ve(B)=A_pc.
\end{equation}
Note that we vectorize $B$ in order to analyze its entries in a compact form, where each entry of $B$ is determined by the corresponding entries of the powers of $A$. Consequently, each entry of $B$ is characterized by a row of $A_p$ through the coefficient vector $c$. 
Let $\boldsymbol{\lambda}_i^T$ denote each row vector of the Vandermonde matrix $\Lambda_p$ in~\eqref{eq:vandermonde}. 
Then, the eigenvalue $\mu_i$ of $B$ can be derived as $\mu_i=\sum_{j=0}^{n-1}c_j\lambda_i^j=\boldsymbol{\lambda}_i^Tc$, while
\begin{equation}\nonumber
    \Lambda(Q)=\{k\lambda_i+(1-k)\boldsymbol{\lambda}_i^Tc\},\quad i\in \until{n}.
\end{equation}
With $\alpha(Q)$ denoting the real part of the right-most eigenvalue of $Q$ as in~\eqref{eq:metric}, $\loneRec$ is equivalent to the following reformulated problem $\rloneRec$: 


\begin{equation}\nonumber
(\rloneRec)\hspace{1mm}
\begin{aligned}
\min_{k, c} \quad & ( \max_i  \Real(k\lambda_i + (1-k)\boldsymbol{\lambda}_i^T c) 
                   + {\|W A_p c\|}_{\ell_1})\\
\textrm{s.t.} \quad 
& 0 \leq k \leq 1, \\
& \ve(I)^T A_p c = \Tr(A).
\end{aligned}
\end{equation}

From $\loneRec$ to $\rloneRec$, we replace $B$ with the vectorized form $A_pc$ according to Lemma~\ref{lemma:1}, and the coefficient vector $c$ is the new decision variable instead. Note that $W\in \mathbb{R}^{n^2\times n^2}$ is a diagonal weighting matrix that introduces the sparsity penalty for each entry of $\ve(B)$. We have that $W=\diag{\ve(\Gamma)}$ and $\|W A_pc\|_{\ell_1}=\|\ve(\Gamma\odot B)\|_{\ell_1}$. Therefore, the term $\|W A_pc\|_{\ell_1}$ is equal to the second term of the objective function in $\loneRec$. 

Next, similar to $\rROS$, we handle the $\max$ operator via a linear program, and reformulate $\rloneRec$ as the following equivalent problem $\rrloneRec$
\begin{equation} \nonumber
(\rrloneRec)\hspace{5mm} 
\begin{aligned}
\min_{k, c,x} \quad & x+ {\|W A_pc\|}_{\ell_1}\\
\textrm{s.t.} \quad 
& 0\leq k \leq 1,\\
& x\geq \Real(k\lambda_i+(1-k)\boldsymbol{\lambda}_i^Tc),\\
& \ve(I)^TA_pc=\Tr(A).\\
\end{aligned}
\end{equation}
The following result establishes the convexity of the reformulated optimization problems $\rloneRec, \rrloneRec$. 

\begin{lemma}[Convexity of the relaxed problem]\label{thm:3}
     $\rloneRec, \rrloneRec$ are biconvex, i.e., $\rloneRec, \rrloneRec$ are convex when fixing $k$ or fixing $c$. 
\end{lemma}
\begin{proof}
The convexity of $\rloneRec, \rrloneRec$ is determined by that of the function $f(k,c)=\max_i \Real(k\lambda_i+(1-k)\boldsymbol{\lambda}_i^Tc)$ since all the constraints are linear. W.l.o.g., assume $\lambda_i\in \mathbb{R}, \forall i$. 
Observe that, when $c$ is fixed, $f(k,c)$ is the composition of two convex functions, thus it is convex. Similarly, 
 $f(k, c)$ is also the composition of two convex functions in $c$, when $k$ is fixed. 
However, $f(k,c)$ is not convex in the pair $(k,c)$, as it involves the bilinear product $kc$. 
\end{proof}

Motivated by the above, we propose two methods to solve the previous problem: one relies on an alternating minimization method, 
and the other on a McCormick relaxation. 

\subsubsection{Alternating minimization} For commutative $A,B\in\mathbb{R}^{n\times n}$, we have $A=P^{-1}U_A P, B=P^{-1}U_B P$ according to~\eqref{eq:upptri}. Therefore, $ U_B = P B P^{-1}$ and we obtain
$\mu_i := (U_B)_{ii} =e_i^\top U_B e_i= e_i^\top (P B P^{-1}) e_i = \Tr( e_i^\top P B P^{-1}e_i)
= \Tr( P^{-1} e_i e_i^\top P B )=\langle E_i^\top,\,B\rangle_F,$ where $e_i\in\mathbb{R}^n$ is the $i$-th standard basis vector, and $E_i=P^{-1} e_i e_i^\top P$. Or equivalently, $\mu_i= (e_i^\top P) B (P^{-1} e_i)
= u_i^\top B v_i$ with $u_i^\top:=e_i^\top P, v_i:=P^{-1} e_i.$ Then, the optimization problem using alternating minimization can be reformulated as follows:

\begin{equation}\nonumber
(\amRec)\hspace{5mm}
\begin{aligned}
\min_{k,x, B} \quad & x+ {\|\Gamma \odot B\|}_{\ell_1}\\
\textrm{s.t.} \quad & AB=BA, \\
& 0\leq k \leq 1, \\
&x\geq k\Real\lambda_i+(1-k) \Real(\langle E_i^\top,\,B\rangle_F), \\
& \Tr(B)=\Tr(A). 
\end{aligned}
\end{equation}
In order to solve $\amRec$, we generate several random initializations of $k \in [0,1]$. For any fixed $k$, the subproblem in $(B,x)$ is convex. Similarly, for any fixed $B$, the subproblem in $(k,x)$ is also convex and can be solved via linear programming. We alternate between these two steps until convergence, i.e., when the relative decrease in the objective falls below a specified tolerance or the maximum number of iterations is reached.  Since each block is solved exactly and each subproblem is convex, the objective function is nonincreasing across iterations and the procedure converges to a coordinate-wise local minimum (stationary point). Among multiple random initializations, we retain the solution with the best objective value. 

\subsubsection{McCormick relaxation}
A McCormick relaxation relaxes $\rrloneRec$ as the following convex optimization; namely, $\cvxRec$, where we introduce a bound $c_i \in [-a, a]$ for each $c_i, i\in \untilzn{n-1}$ with a constant $a>0$. In addition, $\lambda_{A}=[\lambda_1,\cdots,\lambda_n]^T$ and $\mathbf{1}$ denotes the vector of all ones.

  \begin{equation} \nonumber
(\cvxRec)\hspace{5mm}
\begin{aligned}
\min_{k, c,x,w} \quad & x+{\|W A_pc\|}_{\ell_1}\\
\textrm{s.t.} \quad 
& 0\leq k \leq 1,\\
& -a\mathbf{1}\leq c \leq a\mathbf{1},\\
& -ak\mathbf{1}\leq w \leq ak\mathbf{1},\\
& c+(k-1)a\mathbf{1}\leq w \leq c+(1-k)a\mathbf{1},\\
& \mathbf{1}x\geq \Real(\Lambda_p(c-w)+k\lambda_{A}),\\
& \ve(I)^TA_pc=\Tr(A).\\
\end{aligned}
\end{equation}
Clearly, $\cvxRec$ is a convex problem, which can be efficiently solved. Although $\cvxRec$ is a relaxation of $\rrloneRec$ that provides a lower bound of $k$ for $\rrloneRec$ due to McCormick relaxation,  we are still able to find a good sparse $B$ by choosing an appropriate weighting matrix $\Gamma$ or $W$. A McCormick relaxation can be preferable to other approaches due to its efficient implementation, which builds on previous problem structure. 
This relaxation mainly affects the solution of $k$, but we can run Algorithm 1 again to derive the optimal $k$ once the sparse network $B$ is found. We later summarize this procedure and approach for solving Problem~2 in Algorithm~\ref{algorithm3}. 

As we previously discussed, the choice of $\Gamma$ can significantly affect the solution of the network optimization problem. This is shown in Theorem~\ref{thm:1} for problem $\lzeroRec$, where $B^\star=\frac{\Tr(A)}{n}I$.
However,  $B=\frac{\Tr(A)}{n}I$ may be deemed to be an ``easy guess'' by an adversary. Motivated by this, we study here alternative designs for the McCormick relaxation above and heterogeneous $\Gamma$. 
To do this, we start by defining what a desired sparsity patterns is, to provide conditions that ensure the existence of a compatible commutative network.  

\begin{definition} [Sparsity pattern and compatibility]\label{defn:1}
{\rm Consider a network with $n$ nodes, and associated dynamics given by $A\in \mathbb{R}^{n \times n}$. Let $\mathbb{S}\subseteq \{(i,j)\mid 1\leq i, j \leq n\}$ be a set of edges. We say that network $B$ follows the desired sparsity pattern $\mathbb{S}$ if $b_{ij}=0$, $\forall (i,j)\in \mathbb{S}$. Furthermore, we say that $B$ is compatible with $A$ if there exists a nontrivial coefficient vector $c\neq 0$ such that $B \in \mathcal{C}(A)$ under the sparsity pattern $\mathbb{S}$. }\oprocend 
\end{definition}


We now investigate conditions on $\mathbb{S}$ that can ensure that there exists a compatible network $B$ with respect to a given  $A$. According to~\eqref{eq:vec_A2}, it holds that $\ve(B)=A_pc$, where $c$ is the coefficient vector for $B\in \mathcal{C}(A)$. 
Given $\mathbb{S}$,  define the index set $\mathcal{I}_0\subseteq \until{n^2}$ as $\mathcal{I}_0= \{h\mid h=(j-1)n+i, (i,j)\in \mathbb{S}\}$. That is, for a $B$ that follows the sparsity pattern of $\mathbb{S}$, the set $\mathcal{I}_0$ are the indices where the corresponding entries of $\ve(B)$ are zero. Using $\mathcal{I}_0$, define $A_{(r_0)}\in \mathbb{R}^{|\mathbb{S}|\times n}$,
which consists of the rows of $A_p$ 
indexed by $\mathcal{I}_0$.  The following result establishes a connection between the sparsity pattern $\mathbb{S}$, $A_{(r_0)}$, and the existence of compatible network $B$.
\begin{proposition}[Existence of a compatible network]\label{prop:compatibility}

Consider a network $A$ with $n$ nodes, and the desired sparsity pattern $\mathbb{S}$ with $|\mathbb{S}|\geq n$ for $B$. Then, there exists a network $B$ that is compatible with $A$ under sparsity pattern $\mathbb{S}$ if and only if $A_{(r_0)}$ is rank deficient. 
\end{proposition}
\begin{proof}
The existence of $B$ that is compatible with $A$ depends on whether we find a nontrivial coefficient vector $c\neq 0$ that results into $B \in \mathcal{C}(A)$ under the sparsity pattern $\mathbb{S}$. According to Definition \ref{defn:1}, it requires that $b_{ij} = 0$, $\forall (i,j) \in \mathbb{S}$, which specifies that the corresponding entries of $\ve(B)$ must be zero by imposing the condition $A_{(r_0)}c = 0$. Therefore, this is equivalent to finding  conditions on $A_{(r_0)}$ such that the equation $A_{(r_0)}c=0$ has non-trivial solutions; i.e., $c\neq 0$. (Necessity:) If a compatible $B$ exists, then $A_{(r_0)}c = 0$ for a nontrivial $c$. Based on rank–nullity theorem, we have that $\rank(A_{(r_0)})+\nullity(A_{(r_0)})=n$. Since $A_{(r_0)}\in \mathbb{R}^{|\mathbb{S}|\times n} $ with $|\mathbb{S}|\geq n$, and $\nullity(A_{(r_0)})\geq 1$ for a non-trivial $c$, we have that $\rank(A_{(r_0)})\leq n-1$, which leads to that $A_{(r_0)}$ is rank deficient. (Sufficiency:) If  $A_{(r_0)}$ is rank deficient, then $\rank(A_{(r_0)})<n$. Therefore, $\nullity(A_{(r_0)})\geq 1$ and we know that a nontrivial $c$ exists such that $A_{(r_0)}c = 0$, which implies that $b_{ij} = 0$, $\forall (i,j) \in \mathbb{S}$. Furthermore, by construction using the non-trivial coefficient vector $c$, we obtain that $\ve(B)=A_pc$ with $B\in \mathcal{C}(A)$. Therefore, we conclude that there exists a network $B$ that is compatible with $A$ under sparsity pattern $\mathbb{S}$, which completes the proof. 
\end{proof}

Given an initially sparse network $A$, a designer would require a counterpart $B$ that is sparse as well. Thus, in what follows we focus on the case $|\mathbb{S}| \geq n$ in Proposition~\ref{prop:compatibility}, meaning that $B$ should have at least $n$ zero entries to ensure maximal sparsity. In this scenario, $A_{(r_0)}$ is either a square or a tall matrix. 
To obtain a sparse and effective $B$, a natural approach is to select as many rows of $A_p$ as possible, i.e., maximize $|\mathcal{I}_0|$ while ensuring that $A_{(r_0)}$ remains rank-deficient. 
The problem of selecting the maximal number of rows from $A_p$ to form a rank-deficient matrix $A_{(r_0)}$ is in general NP-hard, due to the combinatorial nature of row selection. In the following, we propose an iterative algorithm 
as shown in Algorithm~\ref{algorithm2} to find a sparsity pattern $\mathbb{S}$ such that $|\mathbb{S}|$ is as large as possible, given that there exists a network $B$ compatible with $A$. 

\begin{algorithm}[!ht]
\PrintSemicolon
  \KwInput{$A$ with $n$ nodes}
  \KwOutput{$\mathbb{S}$}

  Compute the basis for the centralizer of $A$ as $I, A, A^2,\cdots, A^{n-1}$\\
  Vectorize $I, A, A^2,\cdots, A^{n-1}$ and derive $A_p=[\ve(I)\cdots\ve(A^{n-1})]$\\
  Initialize: set $i = 0, \mathcal{I}_i = \mathcal{I}_0 \subseteq \until{n^2} $ with $|\mathcal{I}_0|=n$.\\
  Compute $A_{(r_0)}^{(0)}$ from rows of $A_p$ indexed by $\mathcal{I}_0$\\
  \textbf{while} ($A_{(r_0)}^{(0)}$ is full rank) \textbf{do} \\
  \hspace{5mm} Update $\mathcal{I}_0$ by changing the index\\
  \hspace{5mm} Compute $A_{(r_0)}^{(0)}$ from updated $\mathcal{I}_0$\\
  \textbf{end while} \\
  \textbf{while} ($A_{(r_0)}^{(i)}$ is rank deficient) \textbf{do} \\
  \hspace{5mm} Increment $i$\\
  \hspace{5mm} \textbf{repeat}\\  
  \hspace{5mm} Select row index $j$ not in $\mathcal{I}_i$\\
  \hspace{5mm} Update $\mathcal{I}_{i+1} = \mathcal{I}_i \cup \{j\}$\\
  \hspace{5mm} Compute $A_{(r_0)}^{(i+1)}$ based on index set $\mathcal{I}_{i+1}$\\
  \hspace{5mm} \textbf{if} ($A_{(r_0)}^{(i+1)}$ is rank deficient) \textbf{then}\\
  \hspace{10mm} Update $\mathcal{I}_i = \mathcal{I}_{i+1}$\\
  \hspace{10mm} \textbf{break}\\
  \hspace{5mm} \textbf{until} ($A_{(r_0)}^{(i+1)}$ is rank deficient) \\  
  \hspace{5mm} \textbf{if} ($A_{(r_0)}^{(i+1)}$ is full rank) \textbf{then} \\
  \hspace{10mm} \textbf{break}\\
  \textbf{end while} \\
  Derive $\mathbb{S}$ based on the index set $\mathcal{I}_i$\\
  \Return  $\mathbb{S}$
\caption{\textbf{SCNet} --- \textbf{S}parsity Pattern Construction for \textbf{C}ompatible \textbf{Net}works}
\label{algorithm2}
\end{algorithm}

In Algorithm~\ref{algorithm2}, we initialize $A_{(r_0)}^{(0)}$ as a square matrix, where $|\mathbb{S}| = n$ and $n$ entries of the network $B$ are required to be zero. We then iteratively add rows to $A_{(r_0)}^{(0)}$ while ensuring that $A_{(r_0)}^{(i)}$ remains rank deficient. This approach allows the resulting network $B$ to be as sparse as possible.

The construction of $\mathbb{S}$, such that there exists a network $B$ compatible with $A$, is important because it provides guidance on how to select the weighting parameters $\gamma_{ij}$ in $\Gamma$, enabling the optimization problem $\cvxRec$ to yield a sparse network $B$. Specifically, let $\underline{\gamma}$ and $\overline{\gamma}$ be positive weighting parameters, where $\underline{\gamma}$ is significantly smaller than $\overline{\gamma}$, i.e., $\underline{\gamma} \ll \overline{\gamma}$. Based on the resulting sparsity pattern $\mathbb{S}$ obtained from Algorithm~\ref{algorithm2}, we set $\gamma_{ij} = \overline{\gamma}$ if $(i,j) \in \mathbb{S}$, and $\gamma_{ij} = \underline{\gamma}$ if $(i,j) \notin \mathbb{S}$.

With a method for selecting $\Gamma$, we are now prepared to propose an approach using the McCormick relaxation to solve Problem~$2$, as summarized in Algorithm~\ref{algorithm3}. The procedure begins by constructing the basis for the centralizer of $A$ and the Vandermonde matrix $\Lambda_p$ as in~\eqref{eq:vandermonde}. Next, Algorithm~\ref{algorithm2} is run to find $\mathbb{S}$, which is then employed to select the weighting parameters. The matrix $B$ is then obtained solving $\cvxRec$. Finally, Algorithm~\ref{algorithm1} is run to determine $k^\star$ and $\alpha^\star(Q)$.


\begin{algorithm}[!ht]
\PrintSemicolon
  \KwInput{$A$ with $n$ nodes, positive constants $a, \underline{\gamma}, \overline{\gamma}$}
  \KwOutput{$B$ s.t. $AB=BA$, $k^\star,\alpha^\star(Q)$}
  Compute the basis for the centralizer of $A$ as $I, A, A^2,\cdots, A^{n-1}$\\
  Vectorize $I, A, A^2,\cdots, A^{n-1}$ and derive $A_p=[\ve(I)\cdots\ve(A^{n-1})]$\\ 
  Compute $\Tr(A)$ and the eigenvalues of $A$ as $\lambda_i, i\in \until{n}$\\
  Derive $\lambda_{A}=[\lambda_1,\cdots,\lambda_n]^T$ and the Vandermonde matrix $\Lambda_p$ as in~\eqref{eq:vandermonde}\\
  Run Algorithm~\ref{algorithm2} for the desired sparsity pattern $\mathbb{S}$ \\
  Select the weighting matrix $\Gamma=[\gamma_{ij}]$ with $\gamma_{ij} = \overline{\gamma}$ if $(i,j) \in \mathbb{S}$, and $\gamma_{ij} = \underline{\gamma}$ if $(i,j) \notin \mathbb{S}$\\
  Compute the diagonal weighting matrix $W=\diag{\ve(\Gamma)}$\\
  Solve the convex optimization $\cvxRec$ to obtain $B$\\
  Run Algorithm~\ref{algorithm1} for the accurate $k^\star,\alpha^\star(Q)$ \\
  \Return  $k^\star,\alpha^\star(Q), B$
\caption{\textbf{SPNOpt} --- \textbf{S}parsity-\textbf{P}romoting \textbf{N}etwork \textbf{Opt}imization}
\label{algorithm3}
\end{algorithm}
\section{Simulations and examples}\label{sec:simulations}
Here, we present simulations on the scenarios discussed in Section~\ref{sec:application} to validate our previous results. We begin with a small-scale network and a formation control problem, followed by larger-scale numerical and power grid network examples. 

\subsubsection{Resilient formation control}\label{sec:formation-control-example}
We consider a team of five agents in $\mathbb{R}^2$ maintaining a pentagonal formation as they traverse a confined barrier region. We adopt an initial 5-node topology presented in~\cite{liu2022iterative} as a toy illustration. The graph corresponding to this topology $A$ is shown in Fig.~\ref{P1_exmp1_graph} with 
\begin{equation}\nonumber
    A = \begin{bmatrix}
0&1&0&0&0\\
0&0&1&0&0\\
0&0&0&1&0\\
0&0&0&0&1\\
-150&-260&-187&-69&-13
\end{bmatrix}.
\end{equation}


This matrix is given as $A = -(L_w+K)$, where $L_w$ is the weighted Laplacian and $K$ is a gain matrix that modifies the entries of $L_w$. The matrix is then employed as in Section~\ref{sec:application} in a Kronecker product to obtain the corresponding 2D vehicle formation system. The formation is driven by a time‑varying reference signal $g(t)$ that smoothly contracts the pentagon to fit through the narrow barrier region and then re‑expands it. Here, note that all $5$ agents are leaders that have the access to this signal.  The closed-form expression of $g(t)$ is given as follows:
Let $\theta_i=2\pi(i-1)/5, i\in \until{5}$. The total maneuver time  $T=40$ is divided into two phases of durations $T_1=17.144$ and $T_2=T-T_1=22.856$. For each agent $i$, the reference $g_i(t)$ is
\begin{equation}\nonumber
     g_i(t)= 
\begin{cases}
    \left(\begin{smallmatrix}
        t+(1-\phi_1(t))r_w\cos\theta_i+\phi_1(t)r_n\cos\theta_i\\
        (1-\phi_1(t))r_w\sin\theta_i+\phi_1(t)r_n\sin\theta_i
    \end{smallmatrix}\right),& 0\leq t\leq T_1,\\
    \left(\begin{smallmatrix}
        t+(1-\phi_2(\tau))r_n\cos\theta_i+\phi_2(\tau)r_w\cos\theta_i\\
        (1-\phi_2(\tau))r_n\sin\theta_i+\phi_2(\tau)r_w\sin\theta_i
    \end{smallmatrix}\right),& T_1\leq t\leq T,
\end{cases}
\end{equation}
where $\tau=t-T_1, r_n=2, r_w=5$ and $\phi_1=0.5(1-\cos(\pi t/T_1))$, $\phi_2=0.5(1-\cos(\pi (t-T_1)/T_2))$. Stacking these $2$‑dimensional vectors gives the full $g(t)\in \mathbb{R}^{10}$. 
The barrier may represent a region where an adversary is present, and the goal is to find a complementary topology $B$ and an optimal topological switching that can be used along entire path of the formation to increase its resilience. 

The eigenvalues of $A$ are $-2\pm i, 3, -3\pm i$, thus $\alpha(A)=-2$. We run Algorithm~\ref{algorithm3} to obtain a desired commutative network $B$ to enhance resilience by switching between $A$ and $B$. As an intermediate step of Algorithm~\ref{algorithm3}, we execute Algorithm~\ref{algorithm2} to obtain a sparsity pattern $\mathbb{S}$ such that a compatible network $B$ exists. To do this, we initialize $A_{(r_0)}^{(0)}$  according to the sparsity pattern $\mathbb{S}_0=\{(2,4),(2,2),(3,3),(4,4),(5,5)\}$, and $\text{rank}(A_{(r_0)}^{(0)})=4<5$. The resulting sparsity pattern $\mathbb{S}$, obtained via Algorithm~\ref{algorithm2}, is
\begin{multline*}
\mathbb{S}=\{(2,2),(2,3),(2,4),(2,5),(3,1),(3,3),(3,4),(3,5), \\
(4,1),(4,2),(4,4),(4,5),(5,1),(5,2),(5,3),(5,5)\},
\end{multline*}
with $|\mathbb{S}|=16$ and the associated $\text{rank}(A_{(r_0)})=4$. Based on this, we choose a weighting matrix $\Gamma=[\gamma_{ij}]$, where $\gamma_{ij} = \overline{\gamma}=100$ if $(i,j) \in \mathbb{S}$, and $\gamma_{ij} = \underline{\gamma}=1$ if $(i,j) \notin \mathbb{S}$. Next, via Algorithm~\ref{algorithm3} we solve the optimization problem  $\cvxRec$ to obtain the desired network $B$.

The resulting graph is shown as $B$ in Fig.~\ref{P1_exmp1_graph} with
\begin{equation}\nonumber
    B = \begin{bmatrix}
-13&-9.35&-3.45&-0.65&-0.05\\
7.5&0&0&0&0\\
0&7.5&0&0&0\\
0&0&7.5&0&0\\
0&0&0&7.5&0
\end{bmatrix},
\end{equation}
and $\alpha(B) = -2.25$. Since the McCormick relaxation 
affects the solution of $k$, we apply Algorithm 1 to derive the optimal $k^\star$ and the resilience $\alpha^\star(Q)$ of the switched system under the optimal ratio $k^\star$. The resulting values are $k^\star = 0.4286$ and $\alpha^\star(Q) = -2.5714 < \min(\alpha(A), \alpha(B))$. These results indicate that, for this multi‑agent formation control scenario traversing a narrow barrier region, maintaining topology $A$ for a ratio $k^\star$ of each period and topology $B$ for the complementary ratio $1 - k^\star$ effectively enhances the system’s resilience via network switching. Fig.~\ref{formation} shows the agents’ trajectories over a $40 \mathrm{s}$ period, during which we switch between topologies $A$ and $B$ twice: topology $A$ is held for $k^\star$ of the period (i.e. $17.144 \mathrm{s}$) and topology $B$ for $1 - k^\star$ (i.e. $22.856 \mathrm{s}$). We also plot the resilience of the switched system between $A$ and $B$ for different values of the ratio $k$, as shown in Fig.~\ref{resi_switch}, which shows that $\alpha(Q)$ is optimal at $k^\star$.  

\begin{figure}[!h]
\centering
\includegraphics[width=\columnwidth]{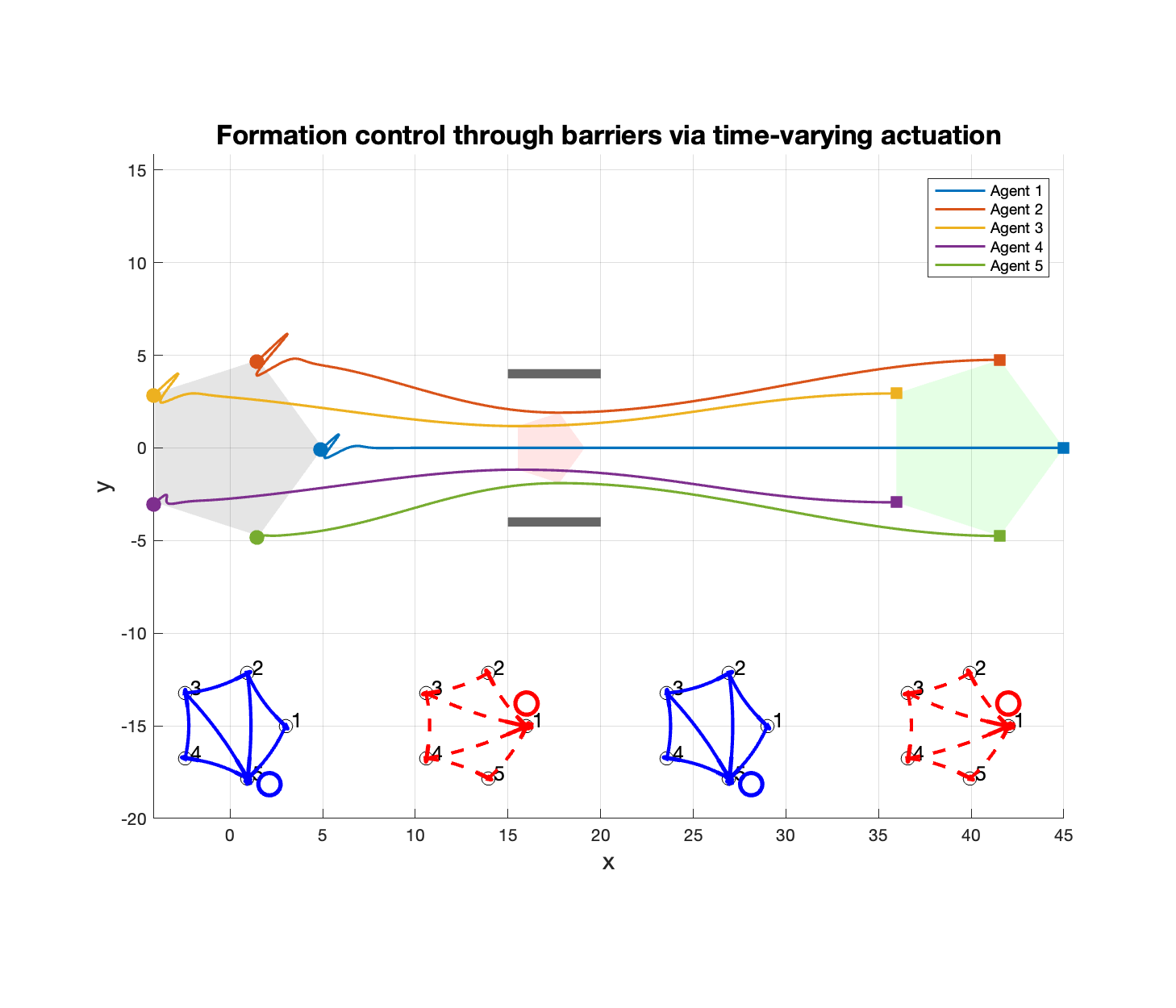}
\caption{Formation control to pass a narrow barrier region via time-varying actuation.}
\label{formation}
\end{figure}

\begin{figure}[!h]
\centering
\includegraphics[width=0.8\columnwidth]{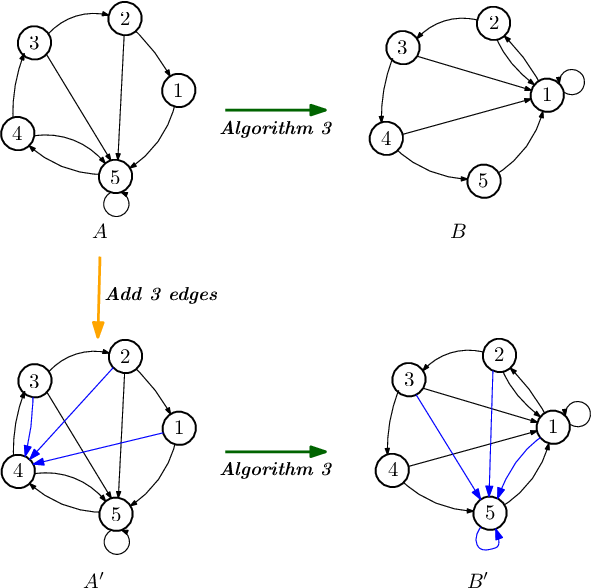}
\caption{The given network topologies (left) and the commutative networks obtained (right) via Algorithm~\ref{algorithm3}.}
\label{P1_exmp1_graph}
\end{figure}

We then add $3$ edges to node $4$ with weight $a_{41}=a_{42}=a_{43}=1$. This modifies the network to $A^\prime$ as shown in Fig.~\ref{P1_exmp1_graph}. We apply Algorithm~\ref{algorithm3} to derive the corresponding complementary network, ensuring that resilience is improved through the time-varying actuation strategy. 
The topology of the updated network is represented as $B^\prime$, where the newly added edges, compared to $B$, include four edges, one of which is a self-loop at node $5$. We compute $\alpha(A^\prime) = -1.1542 $ and $\alpha(B^\prime) = -1.3217$. The optimal switching ratio $k^\star$ and the resilience $\alpha^\star(Q^\prime)$ of the switching system under the optimal ratio $k^\star$ are derived as $k^\star = 0.1809$ and $\alpha^\star(Q^\prime) = -1.91338$, which is less than $\min(\alpha(A^\prime), \alpha(B^\prime))$.

\begin{figure}[!h]
\centering
\includegraphics[width=0.8\columnwidth]{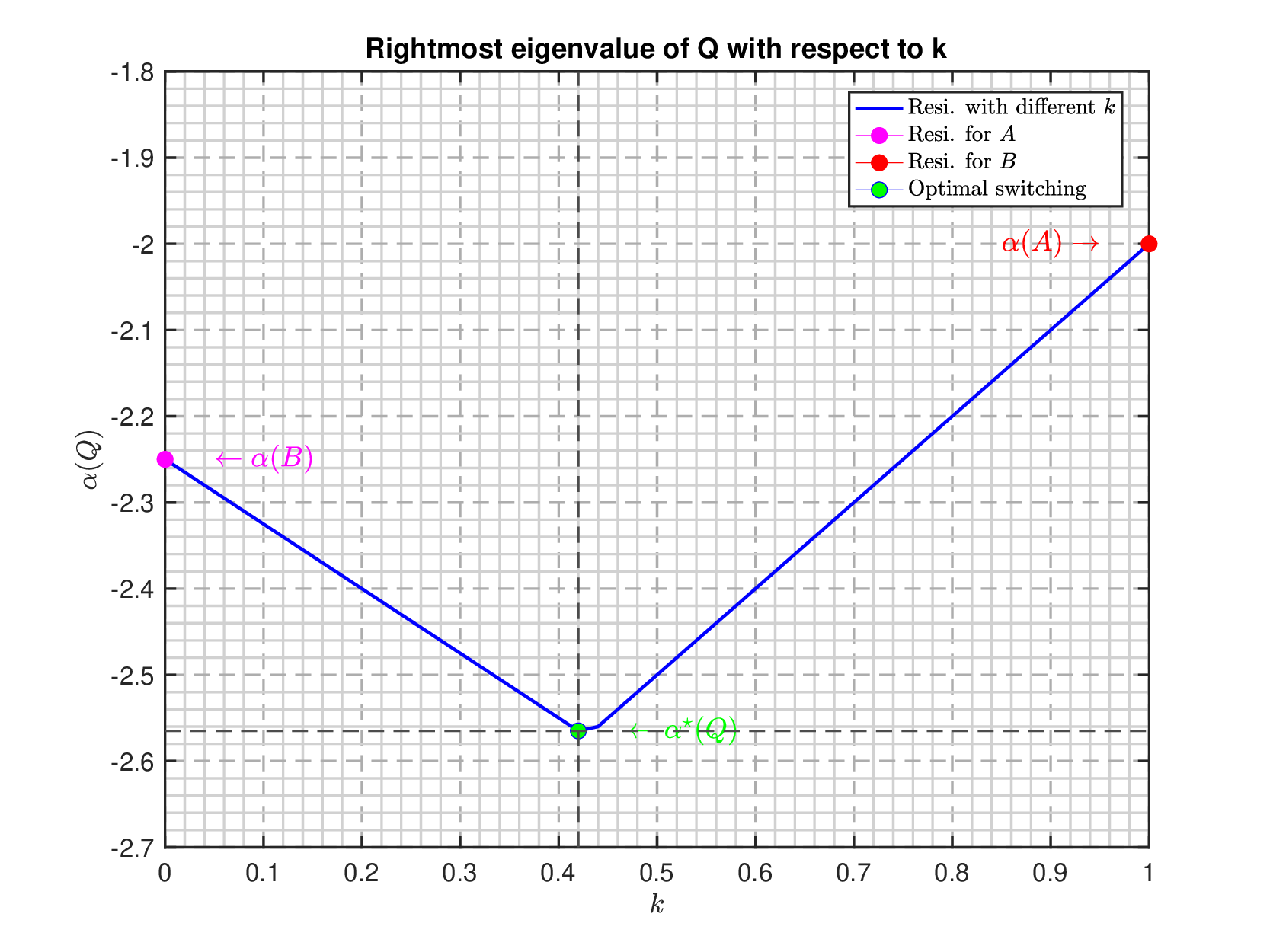}
\caption{Resilience of the switched system for different~$k$.}
\label{resi_switch}
\end{figure}

\subsubsection{High-dimensional numerical example}
We present a numerical simulation of a $100$-node network  to test our algorithms over larger networks. 
The dynamics are represented by a sparse matrix $A\in \mathbb{R}^{100\times 100}$ 
whose sparsity pattern is shown in Fig.~\ref{power_grid}.
As $\alpha(A) = -0.6091$, our goal is to derive a complementary network, $B$, which can improve the resilience of the switched system. In high-dimensional cases such as this, computing the powers of $A^i$ for large $i$ becomes impractical. This problem can be avoided by focusing on a subset of the centralizer of $A$ to construct the commutative network $B$. Another option is to resort to the alternating minimization method, which does not involve matrix powers. 
\begin{figure}[!h]
\centering
\includegraphics[width=\columnwidth, clip, trim=5.6cm 0.5cm 4.4cm 0.2cm]{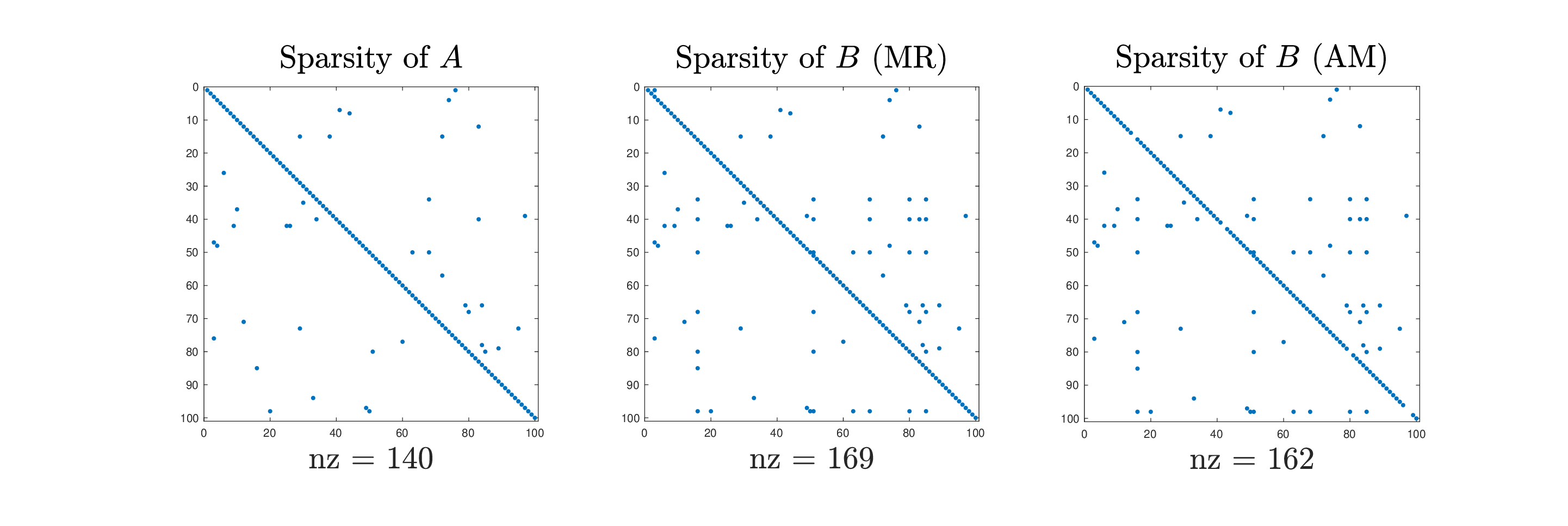}
\caption{Sparsity patterns of the given network $A$ (left, with 140 nonzero entries), the resilience-enhancing complementary network $B$ using McCormick relaxation (middle, with 169 nonzero entries), and the resilience-enhancing complementary network $B$ using alternating minimization (right, with 162 nonzero entries).}
\label{power_grid}
\end{figure}

The given network example has $140$ nonzero entries, thus we explore these options.  Specifically, for the McCormick relaxation method, we first compute powers of $A$ up to order $12$ to construct $B$ via Algorithm~\ref{algorithm3}. 
%
%
The sparsity pattern of $B$ derived using the McCormick relaxation through Algorithm~\ref{algorithm3} is shown in the middle of Fig.~\ref{power_grid}. 
The resilience of network $B$ is derived as $\alpha(B) = -0.3282$, while the resilience of the switched system under optimal switching is $\alpha^\star(Q) = -0.8831$, with the corresponding optimal ratio of $k^\star = 0.3226$. Furthermore, the sparsity pattern of $B$ derived using alternating minimization via the problem~$\amRec$ is shown on the right of Fig.~\ref{power_grid}. The corresponding resilience is $\alpha(B)=0$, while the resilience of the switched system under optimal switching is $\alpha^\star(Q) = -1.2457$, with the corresponding optimal ratio of $k^\star = 0.5472$. 

Thus, we conclude that the resilience of this example can be enhanced by switching between $A$ and $B$; and that the best approach is given by the alternating minimization method. 


\subsubsection{Resilient power network}\label{subsec:power-network-example}
\begin{figure*}[!t]
    \centering
    \includegraphics[width=\textwidth, clip, trim=11cm 10.5cm 9.4cm 10cm]{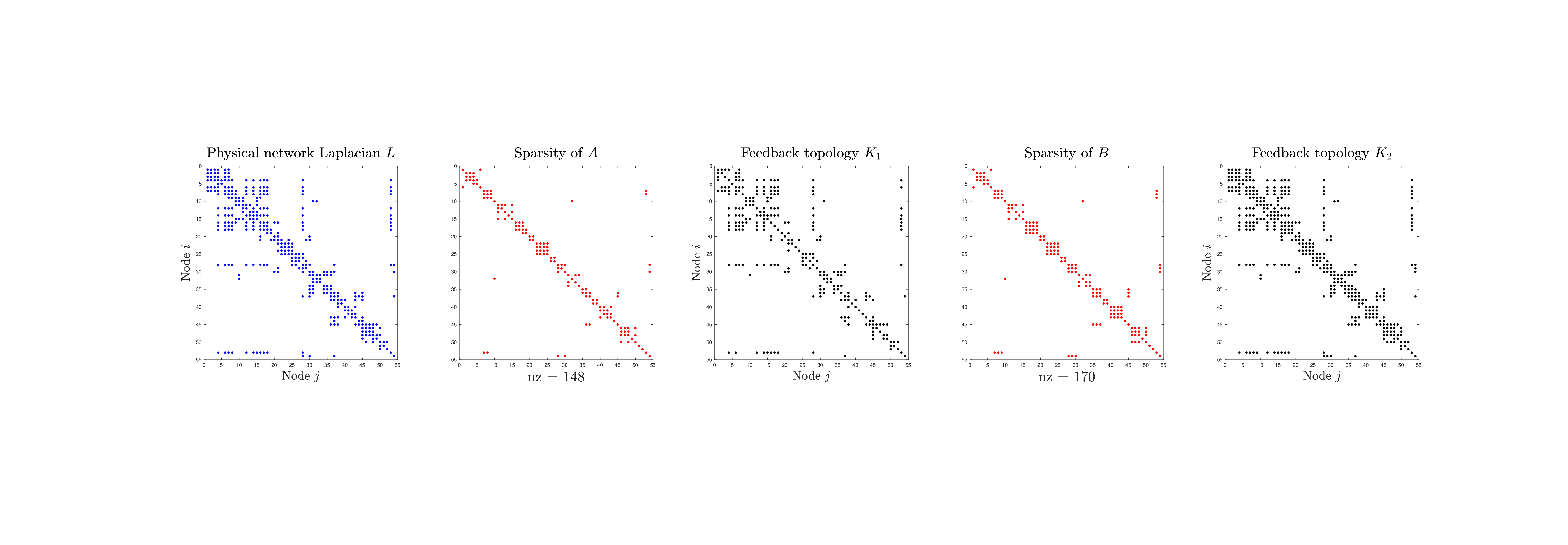}
    \caption{
Sparsity patterns and feedback topologies for the IEEE~118-bus system with $54$ generators. The first subplot (in blue) shows the physical network Laplacian~$L$. 
The second and fourth subplots (in red) depict the sparsity of the closed-loop matrices $A=-(L+K_1)$ and $B=-(L+K_2)$, respectively. The third and fifth subplots (in black) show the corresponding feedback topologies $K_1$ and $K_2$. 
}
\label{fig:bus118}
\end{figure*}
We consider the IEEE 118-bus system with $n=54$ generators. The power network data is obtained from the \textsc{MATPOWER} package~\cite{zimmerman2010matpower}. The physical coupling is described by the generator network Laplacian
$L\in\mathbb{R}^{54\times 54}$, whose given physical topology is shown in the first subplot (in blue) of Fig.~\ref{fig:bus118}. Without loss of generality, we set the damping matrix to the identity, $D=I_n$. We start from a designed feedback design $K_1$, whose topology is shown in the third subplot of Fig.~\ref{fig:bus118}. The corresponding closed-loop matrix is $A= -(L+K_1),$ which is both Hurwitz and sparse with $\alpha(A)=-0.5$. The sparsity pattern of $A$ is depicted in the second subplot of Fig.~\ref{fig:bus118}. We then apply the alternating minimization procedure to enhance resilience since it is more suitable for high-dimensional cases, as shown in the previous numerical example. The sparsity pattern of the resulting resilience-enhancing complementary network $B$ is shown in the fourth subplot of Fig.~\ref{fig:bus118}, with $\alpha(B) = -4.0657$. The corresponding feedback gain is $K_2 = -(L + B)$, and its topology is illustrated in the fifth subplot of Fig.~\ref{fig:bus118}. The resilience of the switched system (via different gain matrices $K_1, K_2$) under optimal switching is $\alpha^\star(Q) = -30.307$, with the corresponding optimal ratio of $k^\star = 0.0975$. Therefore, we can conclude that the resilience is greatly improved through the switching between the gains $K_1$ and $K_2$.


\section{Conclusions}
In this paper, we study how network resilience can be increased via time-varying actuation by switching between pairs of network structures. By prescribing periodic switches, we quantify resilience in terms of the corresponding averaged linear time-invariant counterpart. Then, we study design problems associated with the optimization of a) a switching schedule between two given topologies, and b) the design of a compatible topology and a switching schedule to improve the resilience of a given network. 
Future work includes the study of resilience for non-commutative networks and expanding this design approach to enhance resilience of nonlinear networked systems.

\bibliographystyle{IEEEtran}        
\bibliography{mybib}

\end{document}